\documentclass{article}

\usepackage[margin=1in]{geometry} % Márgenes de 1 pulgada en todos los lados

\usepackage{amsmath}
\usepackage{amsfonts}
\usepackage{amssymb}
\usepackage{graphicx}
\usepackage{xcolor}
\usepackage{cite}
\usepackage[normalem]{ulem}
\providecommand{\U}[1]{\protect\rule{.1in}{.1in}}

\newtheorem{theorem}{Theorem}      % Teoremas numerados de forma independiente
\newtheorem{lemma}{Lemma}      % Numeración independiente

\newtheorem{definition}{Definition}

\newtheorem{remark}{Remark}    % Numeración independiente

\newtheorem{case}{Case}        % Numeración independiente

\newenvironment{proof}[1][Proof]{\noindent\textbf{#1.} }{\ \rule{0.5em}{0.5em}}

\definecolor{navy}{RGB}{0, 0, 128} 
\definecolor{DarkRed}{RGB}{139, 0, 0}

\usepackage{orcidlink}

\begin{document}
%%%%%%%%%%%%%%%%%%%%%%%%%%%%%%%%%%%%%%%%%%%%%%%%%%%%%%%%%%%%%

\title{Strong Hyperbolicity of Second-Order PDEs via Matrix Pencils}

\author{Fernando Abalos\orcidlink{0000-0001-7863-3711}
  \footnote{Email: j.abalos@uib.es} $\,{}^{1}$ and David
  Hilditch\orcidlink{0000-0001-9960-5293}
  \footnote{Email: david.hilditch@tecnico.ulisboa.pt} $\,{}^2$}

\date{\small{${}^1$Departament de F\'isica, Universitat de les Illes
    Balears, Palma de Mallorca, E-07122, Spain and Institute of
    Applied Computing and Community Code (IAC3), Universitat de les
    Illes Balears, Palma de Mallorca, E-07122, Spain,\\ ${}^2$CENTRA,
    Departamento de F\'isica, Instituto Superior T\'ecnico IST,
    Universidade de Lisboa UL, Avenida Rovisco Pais 1, 1049 Lisboa,
    Portugal} }

\maketitle

\begin{abstract}
  We introduce a definition of strong hyperbolicity for second order
  partial differential equations using second order pencils. We show
  that this definition is equivalent to the standard one, derived by
  reducing the equations to first order form, but with the benefit of
  simplifying the calculations necessary to check hyperbolicity. In
  addition, we observe an interesting property, namely that when a
  system is strongly hyperbolic, its second order pencil can be
  factorized as a product of two diagonalizable first order
  pencils. Finally, we present an application to a vector potential
  for of Maxwell's equations, with a general extension and gauge
  fixing.
\end{abstract}

%%%%%%%%%%%%%%%%%%%%%%%%%%%%%%%%%%%%%%%%%%%%%%%%%%%%%%%%%%%%%
\section{Introduction}
%%%%%%%%%%%%%%%%%%%%%%%%%%%%%%%%%%%%%%%%%%%%%%%%%%%%%%%%%%%%%

In physics we frequently encounter systems of nonlinear partial
differential equations (PDEs)~\cite{Ger96}. Basic questions concerning
dynamics within such systems can be addressed by studying their
linearization about arbitrary background solutions. If, for instance,
the linearization of such a system satisfies sufficiently strong
algebraic conditions, then, subject to certain smoothness
requirements, the original PDEs will admit an initial value problem
that is well-posed locally in time. These structural conditions can be
understood by working in a neighborhood of a single point of space and
time small enough that the coefficients of the linearized problem are
constant to a good approximation. Taken together, this is called the
linear, frozen coefficient approximation to the original nonlinear
PDE.

To understand the necessary structural conditions consider, as a
concrete example, the first order linear system
\begin{align}
  \mathfrak{N}^{a}\partial_a\mathbf{u}(x^\mu)=
  \partial_t\mathbf{u}(t,x^i)+\mathbf{A}^p\partial_p\mathbf{u}(t,x^i)
  =\mathbf{f}
  \,,\label{eqn:FT1S}
\end{align}
where~$x^\mu$ denotes the complete collection of coordinates, $t$ the
time coordinate, $x^i$ denotes the spatial coordinates, $\partial_a$
and $\partial_p$ denote the full and spatial partial derivatives
respectively, the matrix coefficients~$\mathbf{A}^p$ are constant
and~$\mathbf{f}$ denotes non-principal (lower-derivative) terms. It is
well known that if the system~\eqref{eqn:FT1S} is strongly hyperbolic,
then its initial value problem is well posed in the~$L^2$
norm~\cite{KreLor89,GusKreOli95,Hil13,Aba17}. The
system~\eqref{eqn:FT1S} is called strongly hyperbolic if, for all unit
spatial covectors~$\hat{k}_p$, the principal
symbol~$\mathbf{A}^{\hat{k}}\equiv\mathbf{A}^p\hat{k}_p$ has real
eigenvalues, a complete set of eigenvectors and obeys the uniformity
condition
\begin{align}
  |\mathbf{T}_{\hat{k}}|+|\mathbf{T}_{\hat{k}}^{-1}|\leq K,
\end{align}
where $K$ denotes a real constant independent of~$\hat{k}_p$, the
matrix~$\mathbf{T}_{\hat{k}}$ is constructed with the eigenvectors as
columns, and~$|\cdot|$ is the usual matrix spectral norm. Systems
which satisfy only the condition on eigenvalues are called weakly
hyperbolic. In this context, well-posedness can be interpreted as the
statement that solutions of~\eqref{eqn:FT1S} to the initial value
problem obey an estimate of the form
\begin{align}
  ||\mathbf{u}(t,\cdot)||_{L^2}\leq C e^{\alpha t}||\mathbf{u}(0,\cdot)||_{L^2}
  \,,
\end{align}
for some real~$C,\alpha>0$. In practice the crucial reason for the
utility of this definition is that it is purely algebraic and, in many
cases, easy to check for a given system of interest.

Following~\cite{AbaReu18}, this definition can equivalently be
reformulated by considering instead the properties of the first order
matrix pencil~$\mathbf{M}(\lambda)=\mathfrak{N}^a(\lambda
n_a+\hat{k}_a)=\lambda\mathbf{1}+\mathbf{A}^{\hat{k}}$ with~$n_a$ a
covector normal to level-sets of constant~$t$, so that~$n_a$
and~$\hat{k}_a$ are not colinear, ensuring a genuine separation
between temporal and spatial directions. We refer to~$\mathbf{M}$ as
the principal pencil (or just the pencil) of the system. Intuitively,
one may then think of the scalar parameter~$\lambda$ as the
pseudodifferential (frequency domain) representation of the time
derivative. This reformulation is useful because it allows for a {\it
  coordinate-free} spacetime notion of strong hyperbolicity, rather
than one tied to a specific time coordinate. In this context it is
natural to consider systems with the more general pencil of the
form~$\mathbf{M}(\lambda)=\mathfrak{N}^a(\lambda
n_a+\hat{k}_a)=\lambda\mathbf{A}^n+\mathbf{A}^{\hat{k}}$,
with~$\mathbf{A}^n$ invertible, in which case the definition of strong
hyperbolicity given above is required to be satisfied
by~$(\mathbf{A}^n)^{-1}\mathbf{A}^{\hat{k}}$. The pencil based
viewpoint is also particularly convenient in the presence of
constraints. A discussion of strong hyperbolicity for constrained
systems in terms of matrix pencils can be found
in~\cite{AbaReu18,Aba22,AbaReuHil24}.

In applications PDEs are often not given in first order form, so it is
important to have a version of this result that holds for high-order
systems that are linear with constant coefficients. For this reason,
definitions of strong-hyperbolicity have been extended to higher-order
derivative systems, see~\cite{NagOrtReu04,GunGar05,HilRic13a} for
examples developed in the context of general relativity (GR). There
are different flavors, but these extended definitions invariably make
use of reduction to first order as a technical crutch, by introducing
variables for {\it all} derivatives of the primitive variables but the
highest. The advantage of formulating definitions of hyperbolicity
directly on the original high-order system is that we then do not have
to deal with the large number of artificial variables and constraints
that arise from a reduction to first order. The core idea behind these
articles is then to identify within the original higher-order system
the piece of the equations of motion that correspond to (a particular
subblock of) the principal symbol of {\it any} first order reduction,
and eventually to establish necessary and sufficient conditions for
well-posedness in a suitable Sobolev norm, which simply corresponds to
the~$L^2$ norm of the state vector of the first order reduction.

There are interesting {\it special cases} in which building a generic
first order reduction by this recipe results in ill-posed PDEs but
where, by instead carefully cherry picking the reduction variables
well-posedness can be obtained. In such cases the original system will
be ill-posed in a Sobolev norm that contains all derivatives as in a
generic first-order reduction but well-posed in a norm with only these
carefully chosen derivatives. For examples in the context of
gravitation, see~\cite{GiaHilZil20,GiaBisHil21,Gun24}.

In this paper we restrict our attention to systems of second order. In
particular, we consider models of the form
\begin{align}
\mathbf{A}^{ab}\partial_a\partial_b\mathbf{u}+\mathbf{f}=0,
\end{align}
assuming that there exist coordinates in which~$\mathbf{A}^{tt}$ is
invertible. We call such a system fully second order. We discuss below
the sense in which fully second order systems are a special case of
first-order-in-time-second-order-in-space systems. Following the
approach just outlined, we consider a pseudodifferential first order
reduction, with the associated first order
pencil~$\mathbf{M}(\lambda)$, similar to the one obtained
from~\eqref{eqn:FT1S}. This reduction avoids the introduction of
artificial constraints that occur when one works in physical space. We
then relate the eigenstructure of~$\mathbf{M}(\lambda)$ to that of the
second-order matrix pencil~$\mathbf{S}(\lambda)$, and finally
characterize hyperbolicity solely in terms
of~$\mathbf{S}(\lambda)$. This turns out to be convenient, as it
substantially reduces the computational effort relative to the
standard hyperbolicity definitions and enables a more compact
presentation of earlier calculations. Interestingly, we find that any
strongly hyperbolic fully second order system has a second order
pencil which can be written as the product of two first order
principal pencils, each of which fulfill the requirement of
strong-hyperbolicity for first order systems. Yet the converse of this
result is false, as the repeated application of strongly-hyperbolic
operators of the form~\eqref{eqn:FT1S} gives rise to fully second
order systems that are only weakly hyperbolic (see
Theorem~\ref{theorem_D2}). To illustrate the utility of these ideas we
apply our results to model problems, culminating with an examination
of the Maxwell equations with potentials and a very general gauge
fixing.

These considerations are of particular interest in the formulation of
theories of gravity. Classic results guarantee well-posedness of the
Cauchy problem in GR, with developments concerning the treatment of
gauge, the initial boundary value problem and numerical
approximation~\cite{Reu04,SarTig12,HilRic13}. Yet because of the
advent of gravitational wave astronomy, there is great interest in
theories whose principal part has a vastly different, more complicated
structure than that of GR. These theories are typically arrived at by
considering changes to the action of GR (for reviews,
see~\cite{DefTsu10, BerBarCar15}). Restricting to second order
derivative field equations, this may be achieved with or without
introducing additional nonminimally coupled fields beyond the metric,
as in Horndeski and Lovelock theories respectively. Higher derivative
theories are also under active consideration. In any case, to make
sense of dynamics a well-posedness result is needed for the initial
value problem for each theory. These results are necessary for
successful computational treatment, and have therefore received much
attention, see for instance~\cite{DelHilWit14, PapRea17, KovRea20,
  KovRea20a, SarBarPre19, ReaDav21, Rea21, ShuAreTha25}. There are too
many individual studies to meaningfully summarize here, but suffice to
say that numerical work is also coming forward in leaps and bounds
with scalar Gauss-Bonnet gravity, which is of second derivative order,
proving particularly popular~\cite{WitGuaPan20, AntLehVen21,
  DonVanYaz22, FraBezBar22, EllHecWit22, DonAreYaz24, ThaFraBez24,
  ThaFraBez25} due to the richness of the phenomenology that it
entails.

For all of this it is helpful both to have a statement of
hyperbolicity that can be checked as easily as possible, but also to
know when the approach taken in GR will fail. Though modest, our
results are useful in both directions. Our calculations for the
Maxwell equations demonstrate the former. Regarding the latter,
consider for instance the results of~\cite{FigHelKov24}, in which the
authors use a careful reduction of high-order operators to along the
lines mentioned in the interesting special cases emphasized above. In
view of our results for second-order systems, it is natural to expect
that, in beyond-GR theories involving repeated applications of
hyperbolic operators, one may need a tailored reduction (such as that
carried out in~\cite{FigHelKov24}) to retain hyperbolicity and thereby
obtain hyperbolic equations of motion.

The paper is structured as follows. In section~\ref{section:SHFORD} we
formulate hyperbolicity conditions on the pencil of a first order in
time reduction of a fully second order system, and introduce the
second order pencil~$\mathbf{S}(\lambda)$. In
Sections~\ref{section:First_Second_Pencils}-\ref{section:Pencil_Diagonalization}
we examine the relationship between the eigenvectors of our first and
second order pencils, observing that the pencils of fully second order
systems naturally decompose as a product of well-behaved first order
pencils. In section~\ref{section:def_strong_hyp_second_ord_1}, we
introduce a useful definition of strong hyperbolicity for second order
PDEs. In section~\ref{section:Applications}. we present applications,
including a treatment of the Maxwell equations with a very general
gauge fixing. We conclude in section~\ref{section:Conclusions}.

%%%%%%%%%%%%%%%%%%%%%%%%%%%%%%%%%%%%%%%%%%%%%%%%%%%%%%%%%%%%%
\section{Strong hyperbolicity at first order in time
  derivatives \label{section:SHFORD}}
%%%%%%%%%%%%%%%%%%%%%%%%%%%%%%%%%%%%%%%%%%%%%%%%%%%%%%%%%%%%%

We consider a manifold~$M=\left(-T,T\right)\times\Sigma$
for~$T\in\mathbb{R}$ with~$\dim M=d$ and~$\dim\Sigma=d-1$. We assume
the existence of coordinates~$x^{a}=\left( t,x^{i}\right)$ such that
spatial coordinates $x^{i}$ are adapted to the hypersurface~$\Sigma$
with the time coordinate~$t$ satisfying~$t\in\left( -T,T\right)$.

We focus on second-order PDEs as
\begin{equation}
  \mathfrak{N}^{ab}\left(x,\phi,\partial\phi\right)  \partial_{a}\partial
  _{b}\mathbf{\phi}+\dots=0, \label{eq_K_1}
\end{equation}
where the ellipses represent lower order derivative terms,
$\mathfrak{N}^{ab}l_{a}l_{b}$ is an~$N\times N$ matrix and~$\phi$
an~$N$-vector. We employ Einstein's summation convention where
repeated indices sum over the natural range implied by the type of
index.

As mentioned in the introduction, we concentrate in this article on
algebraic conditions for the well-posedness of the initial value
problem associated with PDEs such as~\eqref{eq_K_1}. For this
analysis, it is unnecessary to include lower-order terms, allowing us
to focus on a linearized version of the equations,
where~$\mathfrak{N}^{ab}$ can be assumed to be a constant coefficient
matrix. To study the nonlinear setting, additional smoothness
conditions are required, which are not covered in the present
work. For more details about these conditions, see~\cite{KreLor89,
  GusKreOli95, NagOrtReu04}.

Under these assumptions, the system then simplifies
to~$\mathfrak{N}^{ab}\partial_{a}\partial_{b}\phi=0$.
Considering~$n_{a}$ orthogonal to~$\Sigma$ we
have~$n_a\propto-\partial_at$, and the equation can be rewritten as
\begin{align}
  \big(  \mathbf{A}\partial_{t}^{2}+\mathbf{B}^{i}\partial_{i}\partial_{t}
  +\mathbf{C}^{ij}\partial_{i}\partial_{j}\big) \phi=0, \label{eq_1_2_orden_1}
\end{align}
where~$\mathbf{A}\equiv\mathfrak{N}^{ab}n_{a}n_{a}$,
$\mathbf{B}^{i}\equiv2\mathfrak{N}^{ib}n_{b}$
and~$\mathfrak{N}^{ij}\equiv \mathbf{C}^{ij}$. We assume
that~$\mathbf{A}$ is invertible, and call such a system a fully second
order. In addition, we assume initial data
for~$\phi\left(t=0,x\right)$ and~$\partial_{t}\phi( t=0, x )$.

To discuss the hyperbolicity of these equations, we need to reduce
them to first order in derivatives, at least with respect to time. For
this purpose, we define the variable~$\omega$ as
\begin{align}
\partial_t\phi-\omega=0.
\end{align}
The new set of equations can be written as a block matrix–vector
product
\begin{align}
  \left[
    \begin{array}
      [c]{cc}
      \mathbf{1}\partial_{t}
      & -\mathbf{1}\\
      \mathbf{A}^{-1}\mathbf{C}^{ij}\partial_{i}\partial_{j}
      & \mathbf{1}\partial_{t}+\mathbf{A}^{-1}\mathbf{B}^{i}\partial_{i}
    \end{array}
  \right]  \left[
    \begin{array}
      [c]{c}
      \phi\\
      \omega
    \end{array}
  \right]  =0, \label{Eq_dif_fo_1}
\end{align}
where~$\mathbf{1}$ is the~$N\times N$ identity matrix. They define a
complete set of evolution equations for the
variables~$\mathbf{u}=(\phi,\omega)^T$.

Following~\cite{NagOrtReu04}, we now rewrite these equations in a
pseudodifferential form. We take a Fourier transform in the spatial
coordinates, writing the wave vector as~$k_{i}=\left\vert k\right\vert
\hat{k}_{i}$, with~$\hat{k}_{i}$ normalized and~$\left\vert
k\right\vert =\sqrt{\delta^{ij} k_{i}k_{j}}$ the norm ($\delta^{ij}$
is a Riemannian product), and where $\hat{\Phi}=i\left\vert
k\right\vert \hat{\phi }$, we obtain
\begin{align}
  \left[
  \begin{array}
    [c]{cc}
    \mathbf{1}\partial_{t}
    & -i\left\vert k\right\vert \mathbf{1}\\
    i\left\vert k\right\vert \mathbf{A}^{-1}\mathbf{C}^{ij}\hat{k}_{i}\hat{k}_{j}
    & \partial_{t}+i\left\vert k\right\vert \mathbf{A}^{-1}\mathbf{B}^{i}\hat{k}_{i}%
  \end{array}
      \right]  \left[
\begin{array}
  [c]{c}
  \hat{\Phi}\\
  \hat{\omega}
\end{array}
  \right]  =0.
\end{align}

Finally, equation~\eqref{eq_1_2_orden_1} is strongly hyperbolic if the
principal symbol
\begin{align}
i\left\vert k\right\vert \left[
\begin{array}
  [c]{cc}
  0
  & -\mathbf{1}\\
  \mathbf{A}^{-1}\mathbf{C}^{ij}\hat{k}_{i}\hat{k}_{j}
  & \mathbf{A}^{-1}\mathbf{B}^{i}\hat{k}_{i}%
\end{array}
\right]  . 
\end{align}
is diagonalizable with purely complex eigenvalues, and certain
uniformity conditions is satisfied (see~\cite{Kre70, NagOrtReu04,
  HilRic10}).

We aim to provide a definition using matrix pencils. To achieve this,
we can suppress the global factor~$i\left\vert k\right\vert$, define
the first order pencil
\begin{align}
  \mathbf{M}(\lambda) \equiv \left[
  \begin{array}
    [c]{cc}
    \lambda \mathbf{1} & -\mathbf{1}\\
    \mathbf{A}^{-1}\mathbf{C} & \lambda \mathbf{1}+\mathbf{A}^{-1}\mathbf{B}
  \end{array}
              \right], \label{Mat_M_1}
\end{align}
where~$\lambda$ is a parameter, $\mathbf{C}\equiv
\mathbf{C}^{ij}\hat{k}_{i}\hat{k}_{j}$ and~$\mathbf{B}\equiv
\mathbf{B}^{i}\hat{k}_{i}$, and give the following equivalent
definition of strong hyperbolicity.

\begin{definition}
  \label{defn:weak_strong_first_order}
  Eqn.~\eqref{eq_1_2_orden_1} is called weakly hyperbolic with respect
  to~$n_a$ if the associated matrix pencil~$\mathbf{M}(\lambda)$ has
  real eigenvalues. It is called strongly hyperbolic with respect
  to~$n_a$ if we furthermore have:
  
  a) $\mathbf{M}(\lambda)$ is diagonalizable, so that
  \begin{align}
    \mathbf{M}(\lambda) = \mathbf{P}
    ( \lambda \mathbf{1}-\mathbf{D}) \mathbf{P}^{-1}, \label{M_diag_1}
  \end{align}
  for some matrices~$\mathbf{P}$ and~$\mathbf{D}$ of~$2N\times2N$,
  both independent of~$\lambda$, with~$\mathbf{D}$ diagonal and real,
  and with~$\mathbf{1}$ the identity of the same size.

  b) $\left\Vert \mathbf{P}\right\Vert \leq C_{1}$, $\left\Vert
  \mathbf{P}^{-1}\right\Vert \leq C_{2}$, where $C_{1,2}$ are two
  constants independent of~$\hat{k}_{i}$ (Uniformity condition).
\end{definition}

Here we work with a definition for Eqn.~\eqref{eq_1_2_orden_1}
employing the pseudodifferential form of equation~\eqref{Eq_dif_fo_1};
nevertheless, there is an equivalent definition, with the same
conditions, in term of non-pseudodifferential equations obtained by
reducing the system to first-order introducing all derivatives
of~$\phi$ as new variables. See~\cite{GunGar05} for more details.

Observe that a) means that~$\mathbf{M}(0)$ is (Jordan) diagonalizable
with real eigenvalues, and that the introduction of
the~$\lambda$-parameter does not change this diagonalization. This
formulation of the definition is helpful because it permits us to keep
a spacetime (covariant) view even in the frequency domain, albeit with
the cost of working with matrix pencils.

On the other hand, the matrix~$\mathbf{P}=\mathbf{P}(\hat{k}_{i})$ and
its inverse do not depend on~$\left\vert k\right\vert $. Rather,
since~$\mathbf{M}(\lambda)$ depends only on the normalized wave
vector~$\hat{k}_{i}$, so do~$\mathbf{P}$ and~$\mathbf{P}^{-1}$. This
means that~$\hat{k}_{i}$ resides on a sphere (a compact set), then
if~$P$ and~$P^{-1}$ are continuous on~$\hat{k}_{i}$, we know by
topology that their norms have a maximum on this sphere, and the
uniformity condition is satisfied.

To motivate the definition of~$\mathbf{M}(\lambda)$, observe that it
can be deduced from Eqn.~\eqref{Eq_dif_fo_1}, by proposing
\begin{align}
\left[
\begin{array}
	[c]{c}
	\phi\\
	\omega
\end{array}
\right]  =e^{\left(  \lambda t+\hat{k}_{i}x^{i}\right)  }\left[
\begin{array}
	[c]{c}%
	\delta\phi\\
	\delta\omega
\end{array}
\right]
\end{align}
from which we obtain the characteristic equation
\begin{align} 
\mathbf{M}\left(  \lambda\right)  \left[
\begin{array}
	[c]{c}%
	\delta\phi\\
	\delta\omega
\end{array}
\right]  =0. \label{M_1_p_w}
\end{align} 
In summary, we can say that strong hyperbolicity focus on the study of
this equation, solving, for any~$\hat{k}_{i}$, for~$\lambda$,
$\delta\phi$ and~$\delta\omega$.

Finally, to simplify more the discussion, if we forget the
introduction of~$\omega$, the characteristic second order equation can
be obtained substituting~$\phi=e^{\left( \lambda
  t+\hat{k}_{i}x^{i}\right)}\delta\phi$ in
equation~\eqref{eq_1_2_orden_1}. It is
\begin{align}
	\mathbf{S}( \lambda) \delta\phi = 0,  \label{cha_S_1}
\end{align}
with 
\begin{align}
  \mathbf{S}(\lambda)
  \equiv (  \lambda^{2}\mathbf{1}+\lambda \mathbf{A}^{-1}\mathbf{B}
  +\mathbf{A}^{-1}\mathbf{C})  \label{def_S_l_1}
\end{align}
and where we call~$\mathbf{S}(\lambda)$ the second order pencil.

We will see next that solving Eqn.~\eqref{cha_S_1} for~$\lambda$
and~$\delta\phi$ is sufficient to solve~\eqref{M_1_p_w} and conclude
the strong hyperbolicity of system~\eqref{eq_1_2_orden_1}.

%%%%%%%%%%%%%%%%%%%%%%%%%%%%%%%%%%%%%%%%%%%%%%%%%%%%%%%%%%%%%
\section{First and second order pencils}
\label{section:First_Second_Pencils}
%%%%%%%%%%%%%%%%%%%%%%%%%%%%%%%%%%%%%%%%%%%%%%%%%%%%%%%%%%%%%

In this section we collect definitions, lemmas and theorems that we
will rely on afterwards.

\begin{remark}
  \label{remark_0} Considering the first~$N$ rows
  of~$\mathbf{M}(\lambda)$, it is easy to conclude that if~$\det
  \mathbf{M}(\lambda_1)=0$ for some~$\lambda_1$ then the right kernel
  of~$\mathbf{M}\left(\lambda_1\right) $ is forced to take the
  form~$\left[v,\, v\lambda_1\right]^T$, for some $N-$vector~$v\neq0$.
\end{remark}

Following this remark, we call the
pair~$\left(\lambda_1,\left[v_1,\,v_1\lambda_1\right]^T\right)$ an
eigenvalue-eigenvector, or eigenpair, of~$\mathbf{M}\left(
\lambda\right) $ if they satisfy
\begin{align}
  \mathbf{M}\left(  \lambda_1\right)
  \left[
  \begin{array}
    [c]{c}
    v_1\\
    v_1\lambda_1
  \end{array}
  \right]  =0.
\end{align}
This name is inherited since~$\left( -\lambda_1,\left[v_1,\,
  v_1\lambda_1\right]^T\right) $ is a standard eigenpair
of~$\mathbf{M}(0)$.

\begin{remark}
  \label{remark_1} Observe that for any nonvanishing~$v$, and
  any~$\lambda$,
\begin{align}
  \mathbf{M}(\lambda)
  \left[
  \begin{array}
    [c]{c}
    v\\
    v\lambda
  \end{array}
  \right]  =\left[
  \begin{array}
    [c]{c}%
    0\\
\mathbf{S}(\lambda)v
\end{array}
\right],
\end{align}
then, $\left( \lambda_{1},\left[v_1,\,v_1\lambda_1\right]\right) $ is
an eigenpair of~$\mathbf{M}(\lambda)$ if and only if
\begin{align}
  \mathbf{S}(\lambda_1)v_1=0.
\end{align}
\end{remark}

Accordingly, we refer to~$\left(\lambda_1,v_1\right)$ as an eigenpair
of~$\mathbf{S}(\lambda)$ if~$\mathbf{S}(\lambda_1)v_1=0$
and~$v_1\neq0$.

\begin{remark}
  \label{remark_2} We recall that if the first order pencil
  $\mathbf{M}( \lambda) $ has two eigenpairs~$\left(\lambda_1,u_1
  \equiv \left[v_1,\,v_1\lambda_1\right]^T\right)$ and $\left(
  \lambda_2,u_2\equiv\left[v_2,\,v_2\lambda_2\right]^T\right)$
  with~$\lambda_1\neq\lambda_2$ then~$u_{1}$ is linearly independent
  of~$u_2$. This is not necessary true for the second order pencil
  $\mathbf{S}(\lambda)$, which may have eigenpairs~$(\lambda_1,v_1)$
  and~$(\lambda_2,v_2)$ with~$\lambda_1\neq\lambda_2$, where~$v_1$
  and~$v_2$ are linearly dependent.
\end{remark}

\begin{remark}
  \label{remark_det_mat_1} Following~\cite{Ber09} (See the section
  ``Facts on the Determinant of Partitioned Matrices, 2.14.13''),
  when~$\mathbf{A}_{1,2,3,4}$ are~$N\times N$ matrices,
  with~$\mathbf{A}_1\mathbf{A}_2=\mathbf{A}_2\mathbf{A}_1$, then
  \begin{align}
    \det\left(
    \begin{array}
      [c]{cc}
      \mathbf{A}_{1} & \mathbf{A}_{2}\\
      \mathbf{A}_{3} & \mathbf{A}_{4}
    \end{array}\right)
    =\det(\mathbf{A}_{4}\mathbf{A}_{1}-\mathbf{A}_{3}\mathbf{A}_{2})  .
  \end{align}
\end{remark}

We thus conclude that 
\begin{equation}
  \det(\mathbf{M}(\lambda))=\det(\mathbf{S}(\lambda)), \label{detM_detS_1}
\end{equation}
for any value of~$\lambda$.

\begin{remark}
  \label{rem_alg_geo} We recall that~$\mathbf{M}(\lambda)$ is
  diagonalizable if and only if the geometrical and algebraic
  multiplicity of each eigenvalue of~$\mathbf{M}(\lambda)$ coincide.
\end{remark}

The algebraic~$q_{i}$ and geometric~$r_{i}$ multiplicities
of~$\mathbf{M}(\lambda) $ are defined as follows. Denoting
by~$\sigma_{1},...\sigma_{k}$ the distinct eigenvalues
of~$\mathbf{M}(\lambda)$, then
\begin{align}
  r_{i} \equiv \dim\left(\ker \mathbf{M}(\sigma_{i})\right),
\end{align}
and
\begin{equation}
  \det\left(\mathbf{M}(\lambda)\right)  =\left(\lambda-\sigma
    _{1}\right)^{q_{1}}\left(  \lambda-\sigma_{2}\right)^{q_{2}}...\left(
    \lambda-\sigma_{k}\right)  ^{q_{k}}\,. \label{detM_s_1}
\end{equation}
with~$q_1+...+q_k=2N$. For fixed~$i$, each~$q_{i}$ and~$r_{i}$ is
associated with~$\sigma_i$.

On the other hand, following Eqn.~\eqref{detM_detS_1},
\begin{equation}
\det\left(\mathbf{S}(\lambda)\right) = \left( \lambda-\sigma
_{1}\right)^{q_{1}}\left(  \lambda-\sigma_{2}\right)^{q_{2}}\dots\left(
\lambda-\sigma_{k}\right)  ^{q_{k}}\,. \label{detS_s_1}
\end{equation}
This motivates us to also refer to the~$q_{i}$'s as the algebraic
multiplicities of~$\mathbf{S}(\lambda)$ and to
\begin{align}
  s_{i}\equiv\dim\left(\ker\mathbf{S}(\sigma_{i})\right),
\end{align}
as its geometric multiplicities.

As one of our results, we present the following lemma indicating that
the geometric multiplicities (of each eigenvalue)
of~$\mathbf{M}(\lambda) $ and~$\mathbf{S}(\lambda)$ are the same, that
is
\begin{align}
r_{i}=s_{i}\,.
\end{align}

\begin{lemma} \label{lemma_1} For any~$\lambda$,
\begin{equation}
  \dim\left(\ker \mathbf{M}(\lambda)\right)  =\dim\left(\ker\mathbf{S}\left(
      \lambda\right)\right)\,. \label{eq_det_1}
\end{equation}

\end{lemma}

\begin{proof}
  Following Eqn.~\eqref{detM_detS_1}, when~$\lambda\neq\sigma_{i}$,
  then~$0=\dim\left( \ker \mathbf{M}( \lambda) \right) =\dim\left(
  \ker \mathbf{S}(\lambda) \right) =0$. It remains to prove that
\begin{align}
  r_{i}\equiv\dim\left( \ker \mathbf{M}(\sigma_{i})\right)
  =\dim\left(\ker\mathbf{S}(\sigma_{i}) \right)  \equiv s_{i}
\end{align}
for each~$i=1,...,k$. We show this for the eigenvalue~$\sigma_{1}$
(that is, for~$r_{1}=s_{1}$). The proof for the other eigenvalues is
identical.

We recall that following remark~\ref{remark_0}, the elements
of~$\ker(\mathbf{M}(\sigma_{1}))$ has the form~$[v,\, v\sigma_{1}]^T
=[I,\, \sigma_{1}I]^Tv$, and following remarks~\ref{remark_0}
and~\ref{remark_1}, we know that $[I,\,
  \sigma_{1}I]^Tv\in\ker(\mathbf{M}( \sigma_{1})) \Leftrightarrow
v\in\ker(\mathbf{S}(\sigma_{1}))$.

Considering the basis of~$\ker\left( M\left( \sigma_{1}\right)
\right)$ given by
\begin{align}
\textrm{span}\left\langle\left[
\begin{array}
[c]{c}
\mathbf{1}\\
\sigma_{1}\mathbf{1}
\end{array}
\right]  v_{1},\left[
\begin{array}
[c]{c}
\mathbf{1}\\
\sigma_{1}\mathbf{1}
\end{array}
\right]  v_{2},\dots,\left[
\begin{array}
[c]{c}
\mathbf{1}\\
\sigma_{1}\mathbf{1}
\end{array}
\right] v_{r_{1}}\right\rangle =\ker(\mathbf{M}(\sigma_{1})),
\end{align}
with the~$[\mathbf{1},\, \mathbf{1}\sigma_{1}]^Tv_{i}$ linearly
independent, we will show that the vectors~$v_{i}$ are linearly
independent.

We consider a linear combination of the vectors~$[\mathbf{1},\,
  \sigma_{1}\mathbf{1}]^Tv_{i}$, and notice that
\begin{align}
\alpha_{1}\left[
\begin{array}
[c]{c}
\mathbf{1}\\
\sigma_{1}\mathbf{1}
\end{array}
\right]  v_{1}+\dots+\alpha_{r_{1}}\left[
\begin{array}
[c]{c}
\mathbf{1}\\
\sigma_{1}\mathbf{1}
\end{array}
\right]  v_{r_{1}}=\left[
\begin{array}
[c]{c}%
\mathbf{1}\\
\sigma_{1}\mathbf{1}
\end{array}
\right]\left(\alpha_{1}v_{1}+\dots+\alpha_{r_{1}}v_{r_{1}}\right)  .
\label{eq_lc_1}
\end{align}
It is clear that the left hand size vanishes if and only if
\begin{align}
0=\left(  \alpha_{1}v_{1}+...+\alpha_{r_{1}}v_{r_{1}}\right),
\end{align}
but, in addition, the
vectors~$[\mathbf{1},\,\sigma_{1}\mathbf{1}]^Tv_{i}$ are linear
independent, then the LHS vanishes if and only
if~$\alpha_{1}=\dots=\alpha_{r_{1}}=0$. We conclude from these
equivalences that the vectors $v_{1},...,v_{r_{1}}$ are linearly
independent. In addition, since we have~$v_{i}\in\ker
\mathbf{S}(\sigma_{1})$, we conclude that~$r_{1}\leq s_{1}$.

We show now that if~$r_{1}<s_{1}$, we arrive to a contradiction and
finally conclude that~$r_{1}=s_{1}$. If we assume $r_{1}<s_{1}$, then
there exists~$w\in\ker\mathbf{S}(\sigma_{1})$ which is linearly
independent from~$v_{1},\dots,v_{r_{1}}$. Therefore~$[\mathbf{1},
  \sigma_{1}\mathbf{1}]^Tw\in\ker \mathbf{M}(\sigma_{1})$, but
vector~$[ \mathbf{1}, \sigma_{1}\mathbf{1}]w$ and vectors~$[
  \mathbf{1}, \mathbf{1}\sigma_{1}]^Tv_{i}$ can not be linearly
independent since~$\dim(\ker\mathbf{M}(\sigma_{1}))=r_{1}$. We
conclude that there exists~$\beta_{1}=\dots=\beta_{r_{1}+1}$ not all
of them vanishing, such that
\begin{align}
0=\beta_{1}\left[
\begin{array}
[c]{c}
\mathbf{1}\\
\sigma_{1}\mathbf{1}
\end{array}
\right] v_{1}+...+\beta_{r_{1}}\left[
\begin{array}
[c]{c}
\mathbf{1}\\
\sigma_{1}\mathbf{1}
\end{array}
\right] v_{r_{1}}+\beta_{r_{1}+1}\left[
\begin{array}
[c]{c}
\mathbf{1}\\
\sigma_{1}\mathbf{1}
\end{array}
\right]  w
\end{align}
or
\begin{align}
0=\left[
\begin{array}
[c]{c}
\mathbf{1}\\
\sigma_{1}\mathbf{1}
\end{array}
\right]  \left(  \beta_{1}v_{1}+...+\beta_{r_{1}}v_{r_{1}}+\beta_{r_{1}%
+1}w\right)
\end{align}
which means that
\begin{align}
0=\left(  \beta_{1}v_{1}+...+\beta_{r_{1}}v_{r_{1}}+\beta_{r_{1}+1}w\right)
\end{align}
and that~$v_{1},...,v_{r_{1}},w$ are linearly dependent. This is a
contradiction since we started assuming that~$v_{1},...,v_{r_{1}},w$
are linearly independent. The contradiction comes from
assuming~$r_{1}<s_{1}$, so we conclude that~$r_{1}=s_{1}$.

\end{proof}

We have thuse seen that~$\mathbf{M}(\lambda)$
and~$\mathbf{S}(\lambda)$ have the same eigenvalues, algebraic and
geometric multiplicities (see
Eqs.~\eqref{detM_detS_1},\eqref{detM_s_1}
and~\eqref{detS_s_1}). Therefore, recalling remark~\ref{rem_alg_geo},
we conclude with:

\begin{theorem} \label{theo_1} $\mathbf{M}(\lambda)$ is diagonalizable
  if and only if the geometric and algebraic multiplicities of each
  eigenvalue of~$\mathbf{S}(\lambda) $ are equal.
\end{theorem}

This theorem gives the crucial condition in our definition of strong
hyperbolicity for second order PDE (see
definition~\ref{def_strong_hyp_1}). Considering the relationship of
our setup with that for general first-order in time, second-order in
space systems, we have the following.

\begin{remark}
 We can instead write the system in first-order in time, second-order
 in space form
\begin{align}
  \partial_{t}v
  &=\mathbf{A}_{1}^{i}\partial_{i}v+\mathbf{A}_{2}w,
  \label{eqn:FT2S_1}\\
  \partial_{t}w
  &=\mathbf{B}_{1}^{ij}\partial_{i}\partial_{j}v
  +\mathbf{B}_{2}^{i}\partial_{i}w,\label{eqn:FT2S_2}
\end{align}
with~$\mathbf{A}_{2}$ invertible (see~\cite{GunGar05} for details) it
is then possible to show that the pencils
\begin{align}
\mathbf{M}(\lambda)\equiv\left[
\begin{array}
	[c]{cc}
  \lambda \mathbf{1}-\mathbf{A}_{1}^{i}\hat{k}_{i}
  & -\mathbf{A}_{2}\\
  -\mathbf{B}_{1}^{ij}\hat{k}_{i}\hat{k}_{j}
  & \lambda \mathbf{1}-\mathbf{B}_{2}^{i}\hat{k}_{i}%
\end{array}
\right],
\end{align}
and
\begin{align}
  \mathbf{S}(\lambda)\equiv\lambda^{2}\mathbf{1}-\left(  \mathbf{A}_{1}^{i}
  +\mathbf{A}_{2}\mathbf{B}_{2}^{i}\mathbf{A}_{2}^{-1}\right)
  \hat{k}_{i}\lambda+\left(\mathbf{A}_{2}\mathbf{B}_{2}^{i}
  \mathbf{A}_{2}^{-1}\mathbf{A}_{1}^{j}
  -\mathbf{A}_{2}\mathbf{B}_{1}^{ij}\right)  \hat{k}_{i}\hat{k}_{j},
\end{align}
satisfy~\eqref{detM_detS_1}, lemma~\ref{lemma_1} and
theorem~\ref{theo_1}. This can be demonstrated by following similar
steps to those shown here (in other words, setting~$v=\phi$,
and using~\eqref{eqn:FT2S_1} to define~$w$). The definition of strong
hyperbolicity adopted here (see definition~\ref{def_strong_hyp_1}) is
a special case of that of~\cite{GunGar05}, but has the advantage of
being formulated directly in terms of the second-order (principal)
pencil~$\mathbf{S}(\lambda)$.
\end{remark}

%%%%%%%%%%%%%%%%%%%%%%%%%%%%%%%%%%%%%%%%%%%%%%%%%%%%%%%%%%%%%
\section{Pencil Diagonalization} \label{section:Pencil_Diagonalization}
%%%%%%%%%%%%%%%%%%%%%%%%%%%%%%%%%%%%%%%%%%%%%%%%%%%%%%%%%%%%%

In this section, we present the explicit diagonalization of
pencils~$\mathbf{M}(\lambda)$ and~$\mathbf{S}(\lambda)$. For the first
order pencil~$\mathbf{M}(\lambda)$, following Eqn.~\eqref{M_diag_1},
we provide the explicit expression of the~$\mathbf{P}$ matrix. We
furthermore show that the second order pencil~$\mathbf{S}(\lambda)$
can be expressed as the product of two first order pencils. Each of
these two pencils are diagonalizable, with eigenvectors directly
related to~$\mathbf{P}$.

%%%%%%%%%%%%%%%%%%%%%%%%%%%%%%%%%%%%%%%%%%%%%%%%%%%%%%%%%%%%%
\subsection{Diagonalization of $M\left(\lambda\right)$}
%%%%%%%%%%%%%%%%%%%%%%%%%%%%%%%%%%%%%%%%%%%%%%%%%%%%%%%%%%%%%

\begin{remark}
  \label{det_A_B_C_D_1} Consider the square matrix with
  blocks~$\mathbf{A},\mathbf{B},\mathbf{C},\mathbf{D}$. Then
\begin{align}
\det\left[
\begin{array}[c]{cc}
\mathbf{A} & \mathbf{B}\\
\mathbf{C} & \mathbf{D}
\end{array}
             \right]
             =\det( \mathbf{A})
             \det(  \mathbf{D}-\mathbf{C}\mathbf{A}^{-1}\mathbf{B})
             \text{ \ iff
             \ }\mathbf{A} \text{ is invertible}
\end{align}
(see \cite{Ber09}).
\end{remark}

\begin{lemma}
  \label{Lemma_1} $\mathbf{M}( \lambda)$ is diagonalizable if and only
  if
\begin{align}
\mathbf{M}(\lambda)= \mathbf{P}\left[
\begin{array}
[c]{cc}%
\lambda \mathbf{1}-\mathbf{D}_{1} & 0\\
0 & \lambda \mathbf{1}-\mathbf{D}_{2}%
\end{array}
\right]  \mathbf{P}^{-1} \label{Eq_M_P_1}%
\end{align}
with
\begin{align}
\mathbf{P} &  \equiv\left[
\begin{array}
[c]{cc}%
\mathbf{V}_{1} & \mathbf{V}_{1}\mathbf{Q}\\
\mathbf{V}_{1}\mathbf{D}_{1} & \mathbf{V}_{1}\mathbf{Q}\mathbf{D}_{2}%
\end{array}
\right], 
\label{Eq_P_1_a}%\\
\end{align}
$\mathbf{1},\mathbf{V}_{1},\mathbf{Q},\mathbf{D}_{1},\mathbf{D}_{2}$
are~$N\times N$ matrices, $\mathbf{1}$ is the identity,
$\mathbf{D}_{1}$ and~$\mathbf{D}_{2}$ are diagonal, and
\begin{align}
  \det \mathbf{V}_{1} \neq0,  \qquad
  \det\left(\mathbf{Q}\mathbf{D}_{2} - \mathbf{D}_{1} \mathbf{Q}\right)
  \neq0 \label{detQD_1}
\end{align}
\end{lemma}

The definitions of the matrices~$\mathbf{V}_1$ and~$\mathbf{Q}$ are
provided in the following brief proof, and reinterpreted in
theorem~\ref{theor_2}.

\begin{proof}
  $\Rightarrow)$ Following Remark~\ref{remark_0}, if
  $\mathbf{M}( \lambda) $ is diagonalizable then
\begin{align}
\mathbf{P}=\left[
\begin{array}
[c]{c}%
\mathbf{V}\\
\mathbf{V}\mathbf{D}
\end{array}
\right]
\end{align}
where~$\mathbf{V}$ is an~$N\times2N$ matrix,
$\mathbf{D}=\textrm{diag}( \mathbf{D}_{1}, \mathbf{D}_{2})$
with~$\mathbf{D}_{1}$ and~$\mathbf{D}_{2}$ diagonal. The values of
these diagonals are the eigenvalues of~$\mathbf{M}(\lambda)$.

We observe now that for~$\mathbf{P}$ to be invertible, the
matrix~$\mathbf{V}$ must have maximal rank. If this condition is not
satisfied, there exists a vector~$X\neq0$ such
that~$X\mathbf{V}=0$. In that case, we have:

\begin{align}
\left[
\begin{array}
	[c]{cc}%
	\mathbf{X} & 0
\end{array}
\right]  \mathbf{P}=\left[
\begin{array}
	[c]{cc}%
	\mathbf{X} & 0
\end{array}
\right]  \left[
\begin{array}
	[c]{c}%
	\mathbf{V}\\
	\mathbf{V}\mathbf{D}
\end{array}
\right]  =0,
\end{align}
which implies that~$\mathbf{P}$ is not invertible.

Assuming that~$\mathbf{V}$ has maximal rank, we arrange its columns so
that the first~$\mathbf{N}$ columns are linearly independent. This
allows us to express~$\mathbf{V}$ in the form:
\begin{align}
\mathbf{V}=\left[
\begin{array}
[c]{cc}%
\mathbf{V}_{1} & \mathbf{V}_{1}Q
\end{array}
\right],
\end{align}
where~$\mathbf{V}_{1}$ and~$\mathbf{Q}$ are~$N\times N$ matrices and
here~$\det \mathbf{V}_{1}\neq0$. The term~$\mathbf{V}_{1}\mathbf{Q}$
means that the last columns of~$\mathbf{V}$ can be written as a linear
combination of the columns of~$\mathbf{V}_{1}$. Using this structure,
we conclude
\begin{align}
\mathbf{P}  &  \equiv\left[
\begin{array}
[c]{cc}%
\mathbf{V}_{1} & \mathbf{V}_{1}\mathbf{Q}\\
\mathbf{V}_{1}\mathbf{D}_{1} & \mathbf{V}_{1}\mathbf{Q}\mathbf{D}_{2}%
\end{array}
\right]\,.
\end{align}
Finally, by applying remark~\ref{det_A_B_C_D_1}, we find
\begin{align}
  \det \mathbf{P}=\det(\mathbf{V}_{1})^{2}\det(\mathbf{Q}\mathbf{D}_{2}
  -\mathbf{D}_{1}\mathbf{Q}),
\end{align}
implying that the conditions~$\det\mathbf{V}_{1}\neq0$
and~$\det(\mathbf{Q}\mathbf{D}_{2}-\mathbf{D}_{1}\mathbf{Q})\neq0$
must hold to guarantee that~$\mathbf{P}$ is invertible.

$\Leftarrow)$ Holds trivially.
\end{proof}

The following theorem summarizes the foregoing discussion, and
explains how to diagonalize~$\mathbf{M}(\lambda)$ by calculating only
the kernel of~$\mathbf{S}(\lambda)$.

\begin{theorem}
\label{theor_2}
$\mathbf{M}(\lambda)$ is diagonalizable with eigenpairs
\begin{align}%
  \begin{array}
    [c]{ccccc}%
    \left(  \lambda_{i},\operatorname{col}_{i}\left[
    \begin{array}
      [c]{c}%
      \mathbf{V}_{1}\\
      \mathbf{V}_{1}\mathbf{D}_{1}%
    \end{array}
    \right]  \right)
    & \textrm{and}
    & \left(  \rho_{i},\operatorname{col}_{i}\left[
      \begin{array}
        [c]{c}%
        \mathbf{V}_{1}\mathbf{Q}\\
        \mathbf{V}_{1}\mathbf{Q}\mathbf{D}_{2}%
      \end{array}
    \right]  \right)
    &
    & \text{with }i=1,...,N\text{.}
  \end{array}
      \label{M_eigen_pairs}%
\end{align}
where~$\operatorname{col}_{i}[R]$ denotes the~$i$-th column
of~$R$, $\mathbf{V}_{1}$ is invertible and
\begin{align}
  \mathbf{D}_{1}
  &  \equiv\emph{diag}(\lambda_{1},\dots,\lambda_{N}), \label{def_D1_1}
  \\
  \mathbf{D}_{2}
  &  \equiv\emph{diag}(  \rho_{1},\dots,\rho_{N}), \label{def_D2_1}
\end{align}
which implies
\begin{align}
  \det\left(\mathbf{M}(\lambda)\right) = ( \lambda-\rho_{1}) \dots (\lambda-\rho_{N})
  (\lambda-\lambda_{1}) \dots(\lambda-\lambda_{N}), \label{det_M_1}
\end{align}
 if and only if the second order pencil~$\mathbf{S}(\lambda)$
(Eqn.~\eqref{def_S_l_1}) satisfies
\begin{align}
  \det\left( \mathbf{S}(\lambda) \right)  =(\lambda-\rho_{1}) \dots( \lambda-\rho_{N})
  ( \lambda-\lambda_{1}) \dots( \lambda-\lambda_{N}), \label{eq_det_S_1}
\end{align}
with eigenpairs
\begin{align}
\begin{array}
[c]{ccccc}
\left( \lambda_{i},\operatorname{col}_{i}( \mathbf{V}_{1}) \right), &
\text{and} & \left(  \rho_{i},\operatorname{col}_{i}( \mathbf{V}_{1}\mathbf{Q})
\right),  &  & \text{with }i=1,...,N\text{,}%
\end{array}
\label{cod_vec_1}
\end{align}
and finally
\begin{align}
  \det \mathbf{V}_{1}
  & \neq0\label{cond_V1_det_1}\\
  \det(  \mathbf{D}_{1}\mathbf{Q}-\mathbf{Q}\mathbf{D}_{2})
  & \neq0,\label{cond_DQ_det_1}
\end{align}
\end{theorem}

\begin{proof}
  $\Rightarrow)$ Assuming~\eqref{M_eigen_pairs} and using
  Lemma~\ref{Lemma_1} and remark~\ref{remark_1}, we conclude
  \eqref{cod_vec_1}, \eqref{cond_V1_det_1}, \eqref{cond_DQ_det_1}.
  
  $\Leftarrow)$ Assuming Eqn.~\eqref{eq_det_S_1} and using
  Eqn.~\eqref{detM_detS_1}, we conclude Eqn.~\eqref{det_M_1}. This
  means that~$\left\{ \lambda_{i},\rho _{i}\right\}$ are the
  eigenvalues of~$\mathbf{M}(\lambda)$. Assuming
  additionally~\eqref{cod_vec_1}, \eqref{cond_V1_det_1},
  \eqref{cond_DQ_det_1} and using Lemma~\ref{Lemma_1} and
  remark~\ref{remark_1}, we conclude~\eqref{M_eigen_pairs}.
  
\end{proof}

%%%%%%%%%%%%%%%%%%%%%%%%%%%%%%%%%%%%%%%%%%%%%%%%%%%%%%%%%%%%%
\subsection{$\mathbf{S}(\lambda)$ as a product of two first order
  pencils\label{S_M_M_1}}
%%%%%%%%%%%%%%%%%%%%%%%%%%%%%%%%%%%%%%%%%%%%%%%%%%%%%%%%%%%%%

In this subsection we show that when~$\mathbf{M}(\lambda)$ is
diagonalizable the second order pencil~$\mathbf{S}(\lambda)$ is a
product of two first order pencils that are diagonalizable.

\begin{theorem}
  \label{theor_M_S}
  $\mathbf{M}( \lambda)$ is diagonalizable if and only if
  \begin{align}
    \mathbf{S}(\lambda) = \left(  \lambda
    \mathbf{1}-\mathbf{A}_{2}\right) \left(
    \lambda \mathbf{1}-\mathbf{A}_{1}\right) \label{S_Seg_O_1}
  \end{align}
  where
  \begin{align}
    \mathbf{A}_{1} & = \mathbf{V}_{1} \mathbf{D}_{1} \mathbf{V}_{1}^{-1} \label{A_1}\\
    \mathbf{A}_{2} & = \mathbf{V}_{1}(  \mathbf{D}_{1}\mathbf{Q} - \mathbf{Q}\mathbf{D}_{2})
                     \mathbf{D}_{2}(\mathbf{V}_{1}
                     ( \mathbf{D}_{1}\mathbf{Q}-\mathbf{Q}\mathbf{D}_{2}) )^{-1} \label{A_2}
\end{align}
and
\begin{align}
  \det \mathbf{V}_{1}
  & \neq0, \label{eq_det_V1_1}\\
  \det( \mathbf{D}_{1}\mathbf{Q}-\mathbf{Q}\mathbf{D}_{2})
  &\neq0\,. \label{eq_det_D1Q_QD2_1}
\end{align}
This means that
\begin{align}
  \mathbf{A}^{-1}\mathbf{C} &= \mathbf{A}_2 \mathbf{A}_1, \label{eq_AC_1__n}\\
  \mathbf{A}^{-1}\mathbf{B} &=-\left(\mathbf{A}_1+\mathbf{A}_2 \right)\,. \label{eq_AB_1__n}
\end{align}

\end{theorem}

\begin{proof}
  $\Rightarrow)$ We first observe that for any value of~$\lambda$, the
  following holds
  \begin{align}
\left[
\begin{array}
[c]{cc}%
\lambda \mathbf{1}+\mathbf{A}^{-1}\mathbf{B} & \mathbf{1}
\end{array}
\right]  \mathbf{M}(\lambda)  =\left[
\begin{array}
[c]{cc}%
\mathbf{S}(\lambda)  & 0
\end{array}
\right]. \label{eq_Proy_M_1}%
\end{align}

Assuming that~$\mathbf{M}(\lambda)$ is given by Eqn.~\eqref{Eq_M_P_1},
multiplying Eqn.~\eqref{eq_Proy_M_1} by~$\mathbf{P}$ and
setting~$\lambda = 0$, we have
\begin{align}
\left[
\begin{array}
[c]{cc}%
\mathbf{A}^{-1}\mathbf{B} & \mathbf{1}
\end{array}
\right]  \mathbf{M}(0)\,\mathbf{P}=\left[
\begin{array}
[c]{cc}
\mathbf{S}(0) & 0
\end{array}
\right]  \mathbf{P}\,.
\end{align}
The right-hand side can be rewritten as
\begin{align}
  \left[
  \begin{array}
    [c]{cc}%
	\left( \mathbf{A}^{-1}\mathbf{C}\right)  \mathbf{V}_{1} & \left(
	\mathbf{A}^{-1}\mathbf{C}\right)  \mathbf{V}_{1}\mathbf{Q}
\end{array}
\right]
\end{align}
Using Eqs.~\eqref{Eq_M_P_1} and~\eqref{Eq_P_1_a}, the left-hand side
becomes
\begin{align}
\left[
\begin{array}
	[c]{cc}%
	-\left( \mathbf{A}^{-1}\mathbf{B}  \mathbf{V}_{1}+\mathbf{V}_{1}\mathbf{D}_{1}\right)  
	\mathbf{D}_{1}  & -\left(  \mathbf{A}^{-1}\mathbf{B} \mathbf{V}_{1}
	\mathbf{Q}+\mathbf{V}_{1}\mathbf{Q}\mathbf{D}_{2}\right)\mathbf{D}_{2}
\end{array}
\right]
\end{align}
Therefore
\begin{align}
  \left(  \mathbf{A}^{-1}\mathbf{C}\right)  \mathbf{V}_{1}
  & = -\left(  \left(  \mathbf{A}^{-1}\mathbf{B}\right)  \mathbf{V}_{1}
    +\mathbf{V}_{1}\mathbf{D}_{1}\right) \mathbf{D}_{1}, \label{Eq_AV_AVQ_1} \\
  \left(  \mathbf{A}^{-1}\mathbf{C}\right)  \mathbf{V}_{1}\mathbf{Q}
  & = -\left(  \left(  \mathbf{A}^{-1}\mathbf{B}\right)  \mathbf{V}_{1}\mathbf{Q}
    +\mathbf{V}_{1}\mathbf{Q}\mathbf{D}_{2}\right)  \mathbf{D}_{2}, \label{Eq_AV_AVQ_2}
\end{align}
from which we conclude
\begin{align}
  \left( \mathbf{A}^{-1}\mathbf{C}\right) =
  -\left( \left( \mathbf{A}^{-1}\mathbf{B}\right) \mathbf{V}_{1}+\mathbf{V}_{1}
  \mathbf{D}_{1}\right) \mathbf{D}_{1} \mathbf{V}_{1}^{-1}\,. \label{eq_AC_1}
\end{align}
Multiplying Eqn.~\eqref{Eq_AV_AVQ_1} by~$\mathbf{Q}$, equating it with
Eqn.~\eqref{Eq_AV_AVQ_2}, and factorizing~$\mathbf{A}^{-1}\mathbf{B}$,
we obtain
\begin{align}
  \left(  \mathbf{A}^{-1}\mathbf{B}\right)
  =-\mathbf{V}_{1}\left(\mathbf{D}_{1}^{2}\mathbf{Q}-\mathbf{Q}\mathbf{D}_{2}^{2}\right)
  \left(\mathbf{D}_{1}\mathbf{Q}-\mathbf{Q}\mathbf{D}_{2}\right)^{-1}
  \mathbf{V}_{1}^{-1}\,.
\end{align}
Since 
\begin{align}
  \left( \mathbf{D}_{1}^{2}\mathbf{Q}-\mathbf{Q}\mathbf{D}_{2}^{2}\right)
  &=\left(  \mathbf{D}_{1}^{2}\mathbf{Q}-\mathbf{D}_{1}%
    \mathbf{Q}\mathbf{D}_{2}+\mathbf{D}_{1}\mathbf{Q}\mathbf{D}_{2}
    - \mathbf{Q}\mathbf{D}_{2}^{2}\right) \nonumber\\
  &=\left(  \mathbf{D}_{1}\left( \mathbf{D}_{1}\mathbf{Q} -
    \mathbf{Q}\mathbf{D}_{2}\right)
    +\left( \mathbf{D}_{1}\mathbf{Q} - \mathbf{Q}\mathbf{D}_{2}\right)
    \mathbf{D}_{2}\right),
\end{align}
we may conclude that
\begin{align}
  \mathbf{A}^{-1}\mathbf{B} =-\mathbf{V}_{1}\left(  \mathbf{D}_{1}
  +\left(  \mathbf{D}_{1}\mathbf{Q}-\mathbf{Q}\mathbf{D}_{2}\right)
  \mathbf{D}_{2}\left(  \mathbf{D}_{1}\mathbf{Q}
  -\mathbf{Q}\mathbf{D}_{2}\right)  ^{-1}\right)
  \mathbf{V}_{1}^{-1}\,. \label{eq_AB_1}
\end{align}

Substituting this result and Eqn.~\eqref{eq_AC_1}
into~$\mathbf{S}(\lambda)$, we obtain
\begin{align}
  \mathbf{S}( \lambda ) = \left( \lambda^{2}\mathbf{1}-\mathbf{V}_{1}
  \mathbf{D}_{1}^{2}\mathbf{V}_{1}^{-1} +\mathbf{A}^{-1}\mathbf{B}
  \left( \lambda \mathbf{1}
  - \mathbf{V}_{1} \mathbf{D}_{1}\mathbf{V}_{1}^{-1}\right)
  \right)\,.
\end{align}
It is easy to verify that
\begin{align}
  \lambda^{2}\mathbf{1}-\mathbf{V}_{1}\mathbf{D}_{1}^{2}\mathbf{V}_{1}^{-1}
  =\left( \lambda \mathbf{1}+\mathbf{V}_{1}\mathbf{D}_{1}\mathbf{V}_{1}^{-1}\right)
  \left(  \lambda \mathbf{1}-\mathbf{V}_{1}\mathbf{D}_{1}\mathbf{V}_{1}^{-1}\right)  ,
\end{align}
and hence
\begin{align}
  \mathbf{S}( \lambda ) = \left(  
  \lambda \mathbf{1} + \mathbf{V}_{1}\mathbf{D}_{1}\mathbf{V}_{1}^{-1}
  + \mathbf{A}^{-1}\mathbf{B}\right)
  \left(  \lambda \mathbf{1} - \mathbf{V}_{1}\mathbf{D}_{1}\mathbf{V}_{1}^{-1}
  \right)\,.
\end{align}
Finally, substituting the expression for~$\mathbf{A}^{-1}\mathbf{B}$,
we find Eqn.~\eqref{S_Seg_O_1}.

$\Leftarrow)$ Replacing~$\mathbf{A}_{1,2}$ in Eqns.~\eqref{eq_AC_1__n}
and~\eqref{eq_AB_1__n}, we obtain
\begin{align}
  \mathbf{A}^{-1}\mathbf{B}& =-\mathbf{V}_{1}\left(  \mathbf{D}_{1}
    +\left(  \mathbf{D}_{1}\mathbf{Q}-\mathbf{Q}\mathbf{D}_{2}\right)
    \mathbf{D}_{2}\left(  \mathbf{D}_{1}\mathbf{Q}
    -\mathbf{Q}\mathbf{D}_{2}\right)  ^{-1}\right)  \mathbf{V}_{1}^{-1}\\
  \mathbf{A}^{-1}\mathbf{C}& =\mathbf{V}_{1}\left(  \left(  \mathbf{D}_{1}\mathbf{Q}
    -\mathbf{Q}\mathbf{D}_{2}\right)  \mathbf{D}_{2}\left(
    \mathbf{D}_{1}\mathbf{Q}-\mathbf{Q}\mathbf{D}_{2}\right)^{-1}
    \mathbf{D}_{1}\right)  \mathbf{V}_{1}^{-1}
\end{align}
Using these expressions on $\mathbf{M}( \lambda) $ and multiplying on
the right by the matrix
\begin{align}
  \mathbf{P}=\left[
  \begin{array}
    [c]{cc}
    \mathbf{V}_{1} & \mathbf{V}_{1}\mathbf{Q}\\
    \mathbf{V}_{1}\mathbf{D}_{1} & \mathbf{V}_{1}\mathbf{Q}\mathbf{D}_{2}
  \end{array}\right],
\end{align}
we obtain
\begin{align}
\mathbf{M}(\lambda)\mathbf{P}=\mathbf{P}\left[
\begin{array}
  [c]{cc}%
  \lambda \mathbf{1}-\mathbf{D}_{1} & 0\\
  0 & \lambda \mathbf{1}-\mathbf{D}_{2}
\end{array}
\right]\,.
\end{align}
The algebraic and straightforward calculations are left to the
reader. Finally, multiplying by~$\mathbf{P}^{-1}$, we obtain the
diagonalization of~$\mathbf{M}(\lambda)$ and conclude the proof.
\end{proof}

Following the previous theorem, we say that~$\mathbf{S}(\lambda)$ is
diagonalizable with real eigenvalues when~$\mathbf{D}_{1,2}$ are real
and Eqns.~\eqref{S_Seg_O_1}, \eqref{A_1}, \eqref{A_2},
\eqref{eq_det_V1_1} and~\eqref{eq_det_D1Q_QD2_1} hold. Observe that,
substituting the expressions for~$(\mathbf{A}_{1,2})$ from
Eqns.~\eqref{A_1} and~\eqref{A_2}, the second-order pencil can be
written as
\begin{align}
  \mathbf{S}(\lambda)=\mathbf{V}_{2}\,(\lambda
  \mathbf{1}-\mathbf{D}_{1})\,\mathbf{V}_{2}^{-1}\;
  \mathbf{V}_{1}\,(\lambda
  \mathbf{1}-\mathbf{D}_{2})\,\mathbf{V}_{1}^{-1},
\end{align}
where~$\mathbf{V}_{2}\equiv
\mathbf{V}_{1}(\mathbf{D}_{1}\mathbf{Q}-\mathbf{Q}\mathbf{D}_{2})$.
This representation makes explicit that~$\mathbf{S}(\lambda)$
factorizes as the product of two diagonalizable first-order matrix
pencils. We thus obtain a constructive factorization of the
second-order pencil in terms of its eigenstructure.

Assuming~$\mathbf{S}(\lambda) $ is diagonalizable, it is interesting
to verify its eigenpairs explicitly. This verification is provided in
the following Lemma.

\begin{lemma}
  Assuming $\mathbf{S}(\lambda) $ diagonalizable and
  Eqns.~\eqref{def_D1_1}, \eqref{def_D2_1} hold, the eigenpairs of
  $\mathbf{S}(\lambda) $ are
\begin{align}%
\begin{array}
[c]{ccccc}%
\left(  \lambda_{i},\operatorname{col}_{i}\left(  \mathbf{V}_{1}\right)  \right)  &
\text{and} & \left(  \rho_{i},\operatorname{col}_{i}\left(  \mathbf{V}_{1}\mathbf{Q}\right)
\right)  &  & \text{with }i=1,...,N\text{.}%
\end{array}
\label{eigen_pairs_S_1}
\end{align}

\end{lemma}

\begin{proof}
  We notice that for any~$N\times N$ matrix~$\mathbf{R}$,
  $\operatorname{col}_{i}(\mathbf{R})=\mathbf{R}\operatorname{col}_{i}(\mathbf{1})$,
  with~$\mathbf{1}$ the identity matrix. Proceeding by direct
  calculation, we have
  \begin{align}
    \mathbf{S}(\lambda_{i})\operatorname{col}_{i}(\mathbf{V}_{1})   
    &=\mathbf{S}(\lambda_{i})\mathbf{V}_{1}\operatorname{col}_{i}(\mathbf{1})
      \nonumber\\
    & =\mathbf{V}_{1}\left(  \mathbf{D}_{1}\mathbf{Q}-\mathbf{Q}\mathbf{D}_{2}\right)
      \left(\lambda_{i}\mathbf{1}-\mathbf{D}_{2}\right)
      \left(\mathbf{D}_{1}\mathbf{Q}-\mathbf{Q}\mathbf{D}_{2}\right)^{-1}
      \left(\lambda_{i}\mathbf{1}-\mathbf{D}_{1}\right)
      \operatorname{col}_{i}\left(\mathbf{1}\right),
  \end{align}
  since~$\operatorname{col}_{i}(\lambda_{i}\mathbf{1}-\mathbf{D}_{1})=0$
  and~$(\lambda_{i}\mathbf{1}-\mathbf{D}_{1}) \operatorname{col}_{i}(
  \mathbf{1}) =\operatorname{col}_{i}(
  \lambda_{i}\mathbf{1}-\mathbf{D}_{1})=0,$ we conclude that
\begin{align}
  \mathbf{S}(\lambda_{i}) \operatorname{col}_{i}( \mathbf{V}_{1} ) = 0\,.
\end{align}
For the other pair,
\begin{align}
  \mathbf{S}(\rho_{i})\operatorname{col}_{i}( \mathbf{V}_{1}\mathbf{Q} )
  &=\mathbf{S}(\rho_{i})\mathbf{V}_{1}\mathbf{Q}\operatorname{col}_{i}(\mathbf{1})\nonumber\\
  &=\mathbf{V}_{1}( \mathbf{D}_{1}\mathbf{Q}-\mathbf{Q}\mathbf{D}_{2})
    (\rho_{i}I-\mathbf{D}_{2}) ( \mathbf{D}_{1}\mathbf{Q}-\mathbf{Q}\mathbf{D}_{2})^{-1}
    (\rho_{i}\mathbf{1}-\mathbf{D}_{1})  \mathbf{Q}\operatorname{col}_{i}(\mathbf{1})\,.
    \label{eq_SVQ_1}
\end{align}
We observe that
\begin{align}
  &(\rho_{i}\mathbf{1}-\mathbf{D}_{2})
   (\mathbf{D}_{1}\mathbf{Q}-\mathbf{Q}\mathbf{D}_{2})^{-1}
   (\rho_{i}\mathbf{1}-\mathbf{D}_{1})\mathbf{Q}\nonumber\\
  &=(\rho_{i}\mathbf{1}-\mathbf{D}_{2})
    (\mathbf{D}_{1}\mathbf{Q}-\mathbf{Q}\mathbf{D}_{2})^{-1}
    \left(\mathbf{Q}(\rho_{i}\mathbf{1}-\mathbf{D}_{2})
    -(\mathbf{D}_{1}\mathbf{Q}-\mathbf{Q}\mathbf{D}_{2})\right),\nonumber\\
  &=\left((\rho_{i}\mathbf{1}-\mathbf{D}_{2})
    (\mathbf{D}_{1}\mathbf{Q}-\mathbf{Q}\mathbf{D}_{2})^{-1}
    \mathbf{Q}-\mathbf{1}\right)
    (\rho_{i}\mathbf{1}-\mathbf{D}_{2}),\label{eq_VQ_1}
\end{align}
where we added and subtracted the term~$\mathbf{Q}\mathbf{D}_{2}$ in
the first equality and distribute the product in second one.

Again, since~$\operatorname{col}_{i}(
\rho_{i}\mathbf{1}-\mathbf{D}_{2}) =0$
then~$(\rho_{i}\mathbf{1}-\mathbf{D}_{2})\operatorname{col}_{i}(
\mathbf{1}) =0$. Thus, we conclude using this, \eqref{eq_SVQ_1}
and~\eqref{eq_VQ_1} that
\begin{align}
  \mathbf{S}( \rho_{i})  \operatorname{col}_{i}( \mathbf{V}_{1}\mathbf{Q}) = 0
\end{align}

\end{proof}

%%%%%%%%%%%%%%%%%%%%%%%%%%%%%%%%%%%%%%%%%%%%%%%%%%%%%%%%%%%%%
\subsubsection{The~$B=0$ case}
%%%%%%%%%%%%%%%%%%%%%%%%%%%%%%%%%%%%%%%%%%%%%%%%%%%%%%%%%%%%%

We observe that if~$\mathbf{B}=0$ and~$\mathbf{M}(\lambda)$ is
diagonalizable with real eigenvalues (see Eqns.~\eqref{Eq_M_P_1},
\eqref{def_D1_1} and~\eqref{def_D2_1}), then using
Eqn.~\eqref{eq_AC_1}, we find
\begin{align}
  \mathbf{A}^{-1}\mathbf{C}=-\mathbf{V}_{1}\mathbf{D}_{1}^{2}\mathbf{V}_{1}^{-1}\,.
\end{align}
It follows that~$-\mathbf{A}^{-1}\mathbf{C}$ is diagonalizable with
real eingevalues~$\{\lambda_i^{2}\}$. On the other hand,
\begin{align}
  (\lambda^{2}\mathbf{1}+\mathbf{A}^{-1}\mathbf{C})
  &= \mathbf{V}_{1}
    \left(  \lambda \mathbf{1}-\mathbf{D}_{1}^{2}\right)  \mathbf{V}_{1}^{-1}\nonumber\\
  &=\mathbf{V}_{1}
    \left(  \lambda \mathbf{1}+\mathbf{D}_{1}\right)
    \left(  \lambda \mathbf{1}-\mathbf{D}_{1}\right)
    \mathbf{V}_{1}^{-1}, \label{eq_B_0_1} 
\end{align}
from which we easily conclude that~$\mathbf{Q}=\mathbf{1}$. We notice
that~$\mathbf{D}_{1}\mathbf{Q}-\mathbf{Q}\mathbf{D}_{2}=2\mathbf{D}_{1}$
is invertible if and only if~$\lambda_i$ does not vanish. Therefore,
using theorem~\ref{theor_M_S}, $\mathbf{M}(\lambda)$ is diagonalizable
if~$-\mathbf{A}^{-1}\mathbf{C}$ is diagonalizable with non-vanishing
(positive) and real eigenvalues~$\{\lambda_i^{2}\}$. This conclusion
is consistent with the results obtained in~\cite{KreOrt01}.

We now comment on an alternative way to derive Eqn.~\eqref{eq_B_0_1}.
If~$\mathbf{B}=0$, using Eqn.~\eqref{eq_AB_1}, we find
\begin{align}
  \mathbf{D}_{2}=
  -(\mathbf{D}_{1}\mathbf{Q}-\mathbf{Q}\mathbf{D}_{2})^{-1}
  \mathbf{D}_{1}( \mathbf{D}_{1}\mathbf{Q}-\mathbf{Q}\mathbf{D}_{2}).
\end{align}
Substituting this expression into Eqn.~\eqref{eq_AB_1__n} and
subsequently into~\eqref{S_Seg_O_1}, we arrive at
Eqn.~\eqref{eq_B_0_1}.

%%%%%%%%%%%%%%%%%%%%%%%%%%%%%%%%%%%%%%%%%%%%%%%%%%%%%%%%%%%%%%%%%%%%%%%
\subsection{Uniformity condition}\label{Uni_cond}
%%%%%%%%%%%%%%%%%%%%%%%%%%%%%%%%%%%%%%%%%%%%%%%%%%%%%%%%%%%%%%%%%%%%%%%%%

To satisfy the uniformity condition of
definition~\ref{defn:weak_strong_first_order} we need to
bound~$\mathbf{P}$ and~$\mathbf{P}^{-1}$. We notice that~$\mathbf{P}$
(see Eqn.~\eqref{Eq_P_1_a}) can be written as
\begin{align}
\mathbf{P}=\mathbf{P}_{1}\mathbf{P}_{2}\mathbf{P}_{3}\mathbf{P}_{4},
\end{align}
where
\begin{align}
  \mathbf{P}_{1} \equiv \left[
  \begin{array}[c]{cc}
    \mathbf{V}_{1} & 0\\
    0 & \mathbf{V}_{1}
  \end{array}
	\right],\quad
  \mathbf{P}_{2} \equiv \left[
  \begin{array}[c]{cc}
  \mathbf{1} & 0\\
  \mathbf{D}_{1} & \mathbf{1}
  \end{array}
        \right],\quad
  \mathbf{P}_{3}\equiv\left[
  \begin{array}[c]{cc}
    \mathbf{1} & 0\\
    0 & \mathbf{Q}\mathbf{D}_{2}-\mathbf{D}_{1}\mathbf{Q}
  \end{array}
        \right],\quad
  \mathbf{P}_{4}\equiv\left[
  \begin{array}[c]{cc}
    \mathbf{1} & \mathbf{Q}\\
    0 & \mathbf{1}
  \end{array}
        \right]\,.
\end{align}
From which it follows
that~$\mathbf{P}^{-1}=\mathbf{P}_{4}^{-1}\mathbf{P}_{3}^{-1}
\mathbf{P}_{2}^{-1}\mathbf{P}_{1}^{-1}$.  Using the Cauchy-Schwarz
inequality for the spectral norm, we readily derive the following
lemma.

\begin{lemma} \label{lemma_UC} There exist constants~$C_{1,2}$,
  independent of~$\hat{k}_{i}$ such
  that~$\left\Vert \mathbf{P}(\hat{k}_{i})\right\Vert \leq C_{1}$
  and~$\left\Vert \mathbf{P}^{-1}(\hat{k}_{i})\right\Vert \leq C_{2}$,
  for all~$\hat{k}_{i}$, if and only if there are constants~$c_{1,2}$
  with
  \begin{align}
    \left\Vert \mathbf{V}_{1}\right\Vert ,\left\Vert \mathbf{Q}\right\Vert,
    \left\Vert \mathbf{Q}\mathbf{D}_{2}
    -\mathbf{D}_{1}\mathbf{Q}\right\Vert  \leq c_{1},\qquad
    \left\Vert \mathbf{V}_{1}^{-1}\right\Vert ,\left\Vert
    (\mathbf{Q}\mathbf{D}_{2}-\mathbf{D}_{1}\mathbf{Q})^{-1}
    \right\Vert \leq c_{2},\label{eqn:uniformity}
  \end{align}
  for all~$\hat{k}_{i}$, and with~$c_{1,2}$ independent
  of~$\hat{k}_{i}$.
\end{lemma}

%%%%%%%%%%%%%%%%%%%%%%%%%%%%%%%%%%%%%%%%%%%%%%%%%%%%%%%%%%%%%
\section{Strong hyperbolicity at second order in time derivatives}
\label{section:def_strong_hyp_second_ord_1}
%%%%%%%%%%%%%%%%%%%%%%%%%%%%%%%%%%%%%%%%%%%%%%%%%%%%%%%%%%%%%

As discussed above, our goal is to give a formulation of strong
hyperbolicity for fully second order PDEs~\eqref{eq_K_1}, using only
the second order pencil~$\mathbf{S}(\lambda)$. Our definition is
formulated in the constant coefficient case as we focus on the
algebraic conditions, but it can be extended to the nonlinear setting
by imposing additional smoothness conditions. Following
Theorem~\ref{theo_1} and Lemma~\ref{lemma_UC}, the proposed definition
is:

\begin{definition}
  \label{def_strong_hyp_1} The fully second order PDE system~\eqref{eq_K_1},
  with~$\mathfrak{N}^{ab}$ constant and $\mathfrak{N}^{ab}n_an_b$
  invertible, is called strongly hyperbolic if the second order pencil
\begin{equation}
  \mathbf{S}(\lambda) =\lambda^{2}\mathbf{1}
  +\lambda \mathbf{A}^{-1}\mathbf{B}+\mathbf{A}^{-1}\mathbf{C},
\end{equation}
with~$\mathbf{A}\equiv \mathfrak{N}^{ab}n_{a}n_{b}$, $\mathbf{B}\equiv
2\mathfrak{N}^{ic}n_{c}\hat{k}_{i}$ and~$\mathbf{C}\equiv
\mathfrak{N}^{ij}\hat{k}_{i}\hat{k}_{j}$ (here~$\hat{k}_{i}$ is the
normalized wave vector, not colinear with~$n_a$), satisfies the
following properties for all $\hat{k}_{i}\neq0$:
\begin{enumerate}
\item $\mathbf{S}(\lambda)$ has real eigenvalues,

\item The algebraic and geometric multiplicities of these eigenvalues,
defined after Eqn.~\eqref{detS_s_1}, are equal,

\item The uniformity conditions~\ref{eqn:uniformity} hold.
\end{enumerate}  
If in addition to these conditions, the eigenvalues
of~$\mathbf{S}(\lambda)$ are pairwise distinct, system~\eqref{eq_K_1}
is called strictly hyperbolic. On the other hand, if either condition
2 or 3 fail to hold, system~\eqref{eq_K_1} is called weakly
hyperbolic.
\end{definition}

In physical applications the uniformity condition is often, though not
always, verified by virtue of continuity of the eigenvectors on the
normalized wave vector. This is very helpful since this condition is
the trickiest part of the definition to check in practice. The stated
definition is particularly useful for systems written in a covariant
fashion without an explicit~$d+1$ decomposition. Moreover, working
with the matrix pencil is convenient and yields the same results as
either the time reduction or the time-space reduction approach. It is
expected that a version of this definition can be made for yet higher
order systems, but we postpone this for future work.
  
As mentioned in the introduction, we now present a theorem showing
that repeatedly applying the same strongly hyperbolic operator yields
a weakly hyperbolic system.
  
\begin{theorem}\label{theorem_D2}
  Consider the first-order system
  \begin{align}
    \left( \mathbf{1} \partial_t + \mathbf{B}^i \partial_i \right) \phi = 0
  \end{align}
  where~$\mathbf{1}$ is the identity matrix, $\mathbf{B}^i$
  are~$N\times N$ matrices, and~$\phi$ is an~$N$-component vector
  field. Suppose this system is strongly hyperbolic. Then the
  second-order system obtained by applying the same first-order
  operator twice,
  \begin{align}
    \left( \mathbf{1}\partial_t + \mathbf{B}^i \partial_i \right)
    \left( \mathbf{1}\partial_t + \mathbf{B}^i \partial_i \right)\phi = 0
  \end{align}
  is always weakly hyperbolic.
\end{theorem}

We recall that when we say that a system is \emph{weakly hyperbolic},
we mean that, after introducing all spatial derivatives of $\phi$ as
auxiliary variables so as to recast the system in first-order form,
the resulting first-order evolution equations are weakly hyperbolic in
the standard sense.
    
To clarify this conclusion, we present a simple example in
section~\ref{Awe} (see the case~$a=b$), where the theorem’s result is
examined explicitly. The proof of the theorem is brief.
  
 \begin{proof}
   Since the first-order system is strongly hyperbolic, its principal
   symbol satisfies
   \begin{align}
     (\mathbf{1}\,\lambda + \mathbf{B}^{i}\hat{k}_{i})
     = \mathbf{P}(\lambda \mathbf{1} - \mathbf{D})\mathbf{P}^{-1},
   \end{align}
   where~$\mathbf{D} = \mathrm{diag}(\lambda_{1},\ldots,\lambda_{N})$
   has real eigenvalues, $\hat{k}_{i}$ is the wave vector,
   and~$\mathbf{P}=\mathbf{P}(\hat{k})$ is the corresponding
   invertible diagonalization matrix.
   
   The principal symbol of the second-order system is therefore
   \begin{align}
     (\mathbf{1}\,\lambda + \mathbf{B}^{j}\hat{k}_{j})
     (\mathbf{1}\,\lambda + \mathbf{B}^{i}\hat{k}_{i})
     = \mathbf{P}(\lambda \mathbf{1} - \mathbf{D})^{2}\mathbf{P}^{-1},
   \end{align}
   with
   \begin{align}
     (\lambda \mathbf{1} - \mathbf{D})^{2}
     = \mathrm{diag}((\lambda - \lambda_{1})^{2},\ldots,(\lambda - \lambda_{N})^{2}).
   \end{align}
   If all eigenvalues~$\lambda_i$ are distinct, then each eigenvalue
   of the squared symbol has algebraic multiplicity~$2$ but geometric
   multiplicity~$1$, so the system is only weakly hyperbolic. If some
   of the~$\lambda_i$ coincide, the geometric multiplicity remains
   equal to half of the algebraic multiplicity, and the system is
   still weakly hyperbolic.
\end{proof}

%%%%%%%%%%%%%%%%%%%%%%%%%%%%%%%%%%%%%%%%%%%%%%%%%%%%%%%%%%%%%
\section{Applications}
\label{section:Applications}
%%%%%%%%%%%%%%%%%%%%%%%%%%%%%%%%%%%%%%%%%%%%%%%%%%%%%%%%%%%%%

We now we apply the definition of
section~\ref{section:def_strong_hyp_second_ord_1} to different sets of
partial differential equations to demonstrate its utility.

%%%%%%%%%%%%%%%%%%%%%%%%%%%%%%%%%%%%%%%%%%%%%%%%%%%%%%%%%%%%%
\subsection{Almost wave equation in 1+1} \label{Awe}
%%%%%%%%%%%%%%%%%%%%%%%%%%%%%%%%%%%%%%%%%%%%%%%%%%%%%%%%%%%%%

In this subsection, we study the following~$1+1$ PDE for the scalar
function $\phi$,
\begin{align}
  \left(\partial_{t}^{2}-(a+b) \partial_{t}\partial_{x}
  +ab\partial_{x}^{2}\right)  \phi=0, \label{eq_toy_1}
\end{align}
with~$a$, $b$ constant, and initial data~$\phi(t=0,x) =f(x)$,
$\partial_{t}\phi(t=0,x) =g(x)$.

We show that this PDE is strongly hyperbolic when~$a\neq b$ and weakly
hyperbolic when~$a=b$. In addition, following
subsection~\ref{S_M_M_1}, we compute~$\mathbf{Q}$, $\mathbf{D}_{1}$,
$\mathbf{D}_{2}$, $\mathbf{V}_{1}$ and
$\mathbf{QD}_{2}-\mathbf{D}_{1}\mathbf{Q}$. We show that when the
eigenvalues `degenerate' (that is, when $a=b$)
then~$\det( \mathbf{QD}_{2}-\mathbf{D}_{1}\mathbf{Q})=0$, which
implies that the system is weakly hyperbolic.

We start by observing that Eqn.~\eqref{eq_toy_1} can be rewritten as
\begin{align}
  (\partial_{t}-a\partial_{x}) (\partial_{t}-b\partial_{x})\phi=0,
\end{align}
whose associated second-order pencil is
\begin{align}
  \mathbf{S}(\lambda)\equiv(\lambda-ak_{x})(\lambda-bk_{x})\,.
\end{align}
It is follows straightforwardly that the eigenpairs of the system
are~$(\lambda_{1}\equiv bk_{x}, \delta\phi_{1}\equiv1)$
and~$(\rho_{1}=ak_{x}, \delta\phi_{2}=1)$. If we consider the
normalized wave vector then~$k_{x}=\pm1$. For brevity we work below
with~$k_x=1$. The other case is similar.

%%%%%%%%%%%%%%%%%%%%%%%%%%%%%%%%%%%%%%%%%%%%%%%%%%%%%%%%%%%%%
\textbf{The case~$a\neq b$:}
since~$\det(\mathbf{S}(\lambda)) =\mathbf{S}(\lambda) $, if~$a\neq b$
then the algebraic~$q_{i}$ and geometric~$s_{i}$ multiplicities
of~$\mathbf{S}(\lambda)$ are simply
\begin{align}
q_{\lambda_{1}}=1, \quad q_{\rho_{1}}=1,\quad
s_{\lambda_{1}}=1, \quad s_{\rho_{1}}=1\,.
\end{align}
Therefore, condition 2 in Definition~\ref{def_strong_hyp_1} is
satisfied. It is straightforward to verify that conditions 1 and 3 are
also satisfied, and thus the PDE is strongly hyperbolic.

%%%%%%%%%%%%%%%%%%%%%%%%%%%%%%%%%%%%%%%%%%%%%%%%%%%%%%%%%%%%%
\textbf{The case~$a=b$:} when~$a=b$, we
obtain~$\lambda_{1}=\rho_{1}=ak_{x}$ then the algebraic and geometric
multiplicities of~$\mathbf{S}(\lambda)$ change instead to
\begin{align}
q_{\lambda_{1}=\rho_{1}}=2,\quad s_{\lambda_{1}=\rho_{1}}=1.
\label{ejem_alg_geom_1}
\end{align}
Since they are not equal, condition 2 is not satisfied, and the PDE is
weakly hyperbolic.

Following subsection~\ref{S_M_M_1}, it is straightforward to verify
that
\begin{align}
  \mathbf{V}_{1}&=1, &\quad &\mathbf{Q}=1,
  &\quad &\mathbf{V}_{1}\mathbf{Q}=1,
\end{align}
and that
\begin{align}
\mathbf{D}_{2}=\rho_{1}=ak_{x}, &\qquad \mathbf{D}_{1}=\lambda_{1}=bk_{x}.
\end{align}
Therefore
\begin{align}
  \mathbf{Q}\mathbf{D}_{2}-\mathbf{D}_{1}\mathbf{Q}
  =\rho_{1}-\lambda_{1}=\left(a-b\right)k_{x}\,.
\end{align}
We notices
that~$\det(\mathbf{Q}\mathbf{D}_{2}-\mathbf{D}_{1}\mathbf{Q})
=\left(a-b\right) k_{x}\neq0$ $\Leftrightarrow$ $a\neq b$, then as we
mention before, this system is strongly hyperbolic only when $a\neq
b$. In the special case~$b=-a$, we recover the standard wave
equation, which is, of course, strongly hyperbolic.

We observe that when~$a=b$, Eqn.~\eqref{eq_toy_1} reduces to
\begin{align}
  ( \partial_{t}- a \partial_{x} )^{2}\phi=0\,. \label{eq_wh_1}
\end{align}
If we reduce this equation to first order by defining, 
\begin{align}
  u \equiv ( \partial_{t}-a\partial_{x} ) \phi, \label{u_1}
\end{align}
then the final system becomes
\begin{align}
\left[
\begin{array}
  [c]{cc}%
  \left(  \partial_{t}-a\partial_{x}\right)  & 0\\
  0 & \left(  \partial_{t}-a\partial_{x}\right)
\end{array}
\right]  \left[
\begin{array}
  [c]{c}%
  u\\
  \phi
\end{array}
\right]  =\left[
\begin{array}
  [c]{c}%
  0\\
  u
\end{array}
\right]\,.\label{u_phi}
\end{align}
Since this system is diagonal in its principal part, it is strongly
hyperbolic. This appears to contradict our previous statement about
weak hyperbolicity. The key to resolving this apparent contradiction
lies in the type of first order reduction considered and that some
reductions have stronger properties than others.

The reduction~\eqref{u_phi} allows us to build an energy estimate,
see~\cite{Kre70}, that includes~$\phi$ and~$u$. This is not however
sufficient to guarantee that~$\phi$, $\partial_{t}\phi$,
and~$\partial_{x}\phi$ will remain bounded by their initial data over
time and thus to demonstrate well posedness on a norm that contains
all three. For this reason, we do not consider such reductions (such
as \eqref{u_phi}) in our framework. However, in view of
Theorem~\ref{theorem_D2}, it would be interesting to study them
systematically and, in particular, to understand the resulting norms
and the class of source terms that preserve the weaker estimates.

In addition, our definition of strong hyperbolicity accept reductions
to first order where all derivatives of the variables are defined as
new variables, that is,
\begin{align}
  s_{0}\equiv \partial_{t}\phi,
  \quad s_{1}\equiv\partial_{x}\phi, \label{s_0_s_1}
\end{align}
and, as explained earlier, other types of reductions as well.

In particular, we now explicitly analyze the weak hyperbolicity of the
resulting system of PDEs using the reduction~\eqref{s_0_s_1}. The
system is
\begin{align}
\left[
\begin{array}
  [c]{ccc}%
  \partial_{t} & 0 & 0\\
  0 & \left(  \partial_{t}-2a\partial_{x}\right)  & a^{2}\partial_{x}\\
  0 & -\partial_{x} & \partial_{t}\\
  \partial_{x} & 0 & 0
\end{array}
\right] \left[
\begin{array}
  [c]{c}%
  \phi\\
  s_{0}\\
  s_{1}%
\end{array}
\right]  =\left[
  \begin{array}
    [c]{c}%
    s_{0}\\
    0\\
    0\\
    s_{1}%
  \end{array}
  \right], \label{phi_s_0_s_1}
\end{align}
where we observe evolution equations for all variables and a
constraint that can be added to these evolution equations. However, as
can be easily verified, this constraint does not affect the
problematic principal part, the block associated with~$s_0$
and~$s_1$. This block is not diagonalizable, and the system is
therefore classified as weakly hyperbolic. This can alternatively be
confirmed applying to~\eqref{phi_s_0_s_1}, the singular value
decomposition technique described in~\cite{Aba17} or the first order
pencil technique described in~\cite{AbaReu18}.

%%%%%%%%%%%%%%%%%%%%%%%%%%%%%%%%%%%%%%%%%%%%%%%%%%%%%%%%%%%%%
\subsection{Electrodynamics}
%%%%%%%%%%%%%%%%%%%%%%%%%%%%%%%%%%%%%%%%%%%%%%%%%%%%%%%%%%%%%

We consider the Electrodynamic equations
\begin{align}
\nabla_{a}F^{ab}  &  =0,\label{eq_EF_1}\\
\nabla_{\lbrack a}F_{bc]}  &  =0. \label{eq_EF_2}%
\end{align}
over a 4 dimensional space time~$M$, with metric~$g_{ab}$.

We foliate~$M$ by spacelike hypersurfaces~$\Sigma_{t}$
with~$t\in\left[ 0,T\right]$ such that locally~$M=\left( -T,T\right)
\times\Sigma_t$. We define~$n_{a}$ as the normal covector orthogonal
to~$\Sigma_{t}$, and~$h_{ab}=g_{ab}+n_{a}n_{b}$ as the standard
projector onto $\Sigma_{t}$. It satisfies
that~$n^{a}h_{ab}=h_{ab}n^{b}=0$.

We notice that the expression~$n_{b}\nabla_{a}F^{ab}$
($=n_{b}h_{c}^{a}\nabla_{a}F^{cb}$) is a constraints, as it involves
only spacial derivatives~($h_{c}^{a}\nabla_{a}$) of~$F^{ab}$. That is,
if we wish to solve these equations in the free-evolution approach, we
must choose initial data for~$F^{cb}$ such
that~$n_{b}\nabla_{a}F^{ab}=0$ holds initially.

On the other hand, Eqn.~\eqref{eq_EF_2} can be solved locally, using
the vector pontential~$A_a$, as
\begin{align}
F_{bc}\equiv\nabla_{b}A_{c}-\nabla_{c}A_{b}. \label{eq_F_1}%
\end{align}
Therefore, Eqn.~\eqref{eq_EF_1} reduces to four second order equations
for~$A_a$,
\begin{align}
  \mathfrak{E}^{b}\equiv \nabla_{a}F^{ab}
  =\nabla_{a}\left(\nabla^{a}A^{b}-\nabla^{b}A^{a}\right)
  =0. \label{ele_ev_1}
\end{align}

These equations include one constraint and three evolution equations
for the four components of the vector potential. This means that an
additional evolution equation is required for the remaining
component. The process of adding this extra equation is known as gauge
fixing.

Next, following the approach in~\cite{KovRea20a,KovRea20}, we will
perform an extension, see~\cite{AbaReuHil24}, and perform a gauge
fixing. Subsequently, we will show the simplicity of our method for
verifying the strong hyperbolicity of the system, even in degenerate
cases that have not been studied before.

%%%%%%%%%%%%%%%%%%%%%%%%%%%%%%%%%%%%%%%%%%%%%%%%%%%%%%%%%%%%%
\subsection{Extension}
%%%%%%%%%%%%%%%%%%%%%%%%%%%%%%%%%%%%%%%%%%%%%%%%%%%%%%%%%%%%%

We define the tensor
\begin{align}
\hat{G}^{ba}\equiv \hat{g}^{ba}+\hat{f}^{ba},
\end{align}
where~$\hat{f}^{ba}$ is anti-symmetric, and~$\hat{g}^{ba}$ is a
Lorentzian metric for which~$n_{a}$ is timelike.

In order to give an evolution to the
constraint~$n_{b}\nabla_{a}F^{ab}$, we perform an extension of the
system modifing the equations~$\mathfrak{E}^{a}=0,$ as
\begin{align}
\mathfrak{E}^{b}+\hat{G}^{ba}\nabla_{a}Z+f_{1}^{b}\left(  Z\right)  =0.
\label{Ext_1}
\end{align}
Here~$Z$ is a new variable in our system and~$f^{a}\left( Z\right)$ is
a vector function of~$Z$, such that~$f_{1}^{a}\left( 0\right) =0$. In
general extensions, the function~$f_{1}^{a}\left(Z\right)$ is
typically used to include damping terms.

We observe that if~$Z$ initially vanishes and remains zero throughout
the entire evolution, this extension produces solutions to the
original system~\eqref{ele_ev_1}. We will show this happens in the
following analysis.
 
The variables are now~$A_{a}$ and~$Z$, we need to give initial data
for them over~$\Sigma_{0}$. The initial conditions are~$\left.
Z\right\vert _{\Sigma_{0}}=0$ and some value for~$A_b$ such that the
constraint~$\left. n_{b}\mathfrak{E}^{b}\right\vert _{\Sigma_{0}}=0$
is satisfied initially. Using this, we observe that
\begin{align}
  0=\left. n_{b}\mathfrak{E}^{b}+n_{b}\hat{G}^{ba}\nabla_{a}Z+n_{b}
  f_{1}^{b}\left( Z\right) \right\vert _{\Sigma_{0}}
  =n_{b}\hat{g}^{ba}n_{a} \left.   \left( n^{c}\nabla_{c}Z\right)
  \right\vert_{\Sigma_{0}},
\end{align}
from which we conclude that~$\left.\left( n^{c}\nabla_{c}Z\right)
\right\vert _{\Sigma_{0}}=0$.

On the other hand, if we take the divergence of (\ref{Ext_1}), and
recall that~$\nabla_{b}\nabla_{a}F^{ab}=0$ holds for any value
of~$F^{ab}$, we conclude that
\begin{align}
\hat{g}^{ba}\nabla_{b}\nabla_{a}Z+\left(  \nabla_{b}\hat{G}^{ba}\right)
\nabla_{a}Z+\nabla_{b}f_{1}^{b}\left(  Z\right)  =0\,.
\end{align}

We observe that the anti-symmetric part of~$\hat{G}^{ba}$ does not
play a roll in the principal symbol of this equation. Moreover,
since~$\hat{g}^{ba}$ is Lorentzian, this equation is a wave equation
up to lower-order terms and is therefore strongly hyperbolic and thus
has a well-posed initial value problem. In particular, given suitable
initial data, it admits a unique solution (at least locally in
time). The solution is~$Z=0$.

%%%%%%%%%%%%%%%%%%%%%%%%%%%%%%%%%%%%%%%%%%%%%%%%%%%%%%%%%%%%%
\subsection{Gauge fixing}
%%%%%%%%%%%%%%%%%%%%%%%%%%%%%%%%%%%%%%%%%%%%%%%%%%%%%%%%%%%%%

Equation~\eqref{Ext_1} has a gauge freedom, the
term~$n^{a}n^{b}\nabla_{a}\nabla_{b}\left( n.A\right) $ is not present
on it, so that there is not an evolution equation for~$n.A$. To fix
this gauge, we propose
\begin{align}
  \tilde{G}^{cd}\nabla_{c}A_{d}+f_{2}\left(  A_{b}\right)  =0.
  \label{eq_gauge_1}
\end{align}
Here~$f_{2}$ is a function of~$A_{a}$, which does not include
derivatives, the tensor~$\tilde{G}^{cq}$ is defined as
\begin{align}
\tilde{G}^{cq}\equiv\tilde{g}^{cq}+\tilde{f}^{cq},
\end{align}
where~$\tilde{f}^{cq}$ is anti-symmetric and~$\tilde{g}^{cq}$ is a
Lorentzian metric, for which~$n_a$ is timelike.

It is not hard to verify that this equation indeed provides the extra
evolution equation
\begin{align}
  0  &  =\left(  n_{c}\tilde{g}^{cd}n_{d}\right)
  \left( n^{a}\nabla_{a}\left( n.A \right)  \right) \nonumber\\
  &  +\left(  -\tilde{G}^{cd}n_{d}D_{c}\left(  n.A\right)
  - n_{c} \tilde{G}^{cd}n^{a}\nabla_{a}\left( h_{d}^{b}A_{b} \right)
  +\tilde{G}^{cd}D_{c}\left( h_{d}^{b}A_{b}\right) \right) + \dots 
\end{align}
where the ellipses are lower order terms and~$D_a\equiv
h_a^b\nabla_b$. In addition, using the previous initial data
for~$A_b$, it says the value of $n^{a}\nabla_{a}\left( n.A\right)$
over~$\Sigma_0$.

Finally, in order to make disappear~$Z$ and obtain only second order
evolution equation for~$A_b$, we couple the extension with the gauge
fixing as
\begin{align}
Z\equiv\tilde{G}^{cd}\nabla_{c}A_{d}+f_{2}\left(A_{b}\right) =0\,.
\end{align}
We notice that this is coherent since, as we showed before, $Z$
vanishes for~$t\geq0$.

The final system is 
\begin{align}
  0=\nabla_{a}\left(  \nabla^{a}A^{b}-\nabla^{b}A^{a}\right)
  +\hat{G}^{ba}\nabla_{a}\left(  \tilde{G}^{cd}\nabla_{c}A_{d}+f_{2}\right)
  + f_{1}^{b}\left(  Z\right)\,.  \label{ext_gau_1}
\end{align}
These are four second order evolution equations for the four component
of~$A_b$.

%%%%%%%%%%%%%%%%%%%%%%%%%%%%%%%%%%%%%%%%%%%%%%%%%%%%%%%%%%%%%
\subsection{Strong hyperbolicity}
%%%%%%%%%%%%%%%%%%%%%%%%%%%%%%%%%%%%%%%%%%%%%%%%%%%%%%%%%%%%%

In this subsection, we apply Definition~\ref{def_strong_hyp_1} of
strong hyperbolicity to Eqn.~\eqref{ext_gau_1}. We divide the analysis
into three parts. First, we introduce the second-order
pencil~$S(\lambda)$. Second, we examine the strong hyperbolicity of
the system and explicitly present its eigenvectors. Finally, we
discuss the diagonalization of the principal symbols in both first and
second order formulations.

%%%%%%%%%%%%%%%%%%%%%%%%%%%%%%%%%%%%%%%%%%%%%%%%%%%%%%%%%%%%%
\subsubsection{Second order pencil}
%%%%%%%%%%%%%%%%%%%%%%%%%%%%%%%%%%%%%%%%%%%%%%%%%%%%%%%%%%%%%

We notice that Eqn.~\eqref{ext_gau_1} can be written as
\begin{align}
  \left(\nabla_{a}\nabla^{a}A^{b}-\nabla_{a}\nabla^{b}A^{a}\right)
  +\hat{G}^{ba}\tilde{G}^{cd}\nabla_{a}\nabla_{c}A_{d}+\dots=0,
  \label{eq_Ele_1}
\end{align}
where the ellipses denote lower order derivative terms.

We define standard coordinate adapted to the foliation, as described
in~\cite{Alc08, Gou07, BonPalBon05}. We then perform a 3+1
decomposition to cast the (already) linearized equations in the form
of Eqn.~\eqref{eq_1_2_orden_1} and to construct the associated
second-order pencil~$\mathbf{S}(\lambda)$. Because hyperbolicity is
assessed pointwise, it suffices to work in a local frame at each
spacetime point~$p$. In particular, one may choose a locally inertial
(orthonormal) frame at~$p$ in which the background metric takes the
Minkowski form. In the calculations that follow, one may dispense the
use of a locally inertial frame and extend the analysis
straightforwardly to an arbitrary prescribed Lorentzian background,
but we avoid this here in order to simplify the discussion.
	
With these choices, the coordinate time vector can be taken to
coincide with the unit normal to the slices, so that
\begin{align}
  \partial_{t}\equiv n^{b}\partial_{b}.
\end{align}
It then follows that any derivative admits the decomposition
\begin{align}
  \partial_{b} =-\,n_{b}\,\partial_{t}+h_{b}^{c}\,\partial_{c},
\end{align}
As explained in Section~\ref{section:SHFORD}, to derive the
characteristic equation we consider high-frequency
perturbations of the form $A_{q}=e^{-\lambda
  t+k_{i}x^{i}}\,\delta A_{q}$, where~$k_i$ is normalized,
orthogonal to~$n_a$, and spacelike with respect
to~$g^{ab}$. This procedure is equivalent to the replacement
\begin{align}
  \partial_{b}\rightarrow l_{b}\equiv n_{b}\lambda+k_{b}.
  \label{eq_L_1}
\end{align}

Consequently, disregarding all lower-order terms, we obtain the
desired covariant characteristic equation
\begin{align}
  P^{bq}\delta A_{q}=\left(  \left(  l^{2}g^{qb}-g^{qc}l_{c}l^{b}\right)
  +\hat{G}^{ba}l_{a}\tilde{G}^{cq}l_{c}\right)  \delta A_{q}=0 \label{eq_P_1}%
\end{align}
where~$P^{bq}$ is the associated second order principal symbol.

To compute~$\mathbf{S}(\lambda)$, as defined in
Eqn.~\eqref{def_S_l_1}, we notice that
\begin{align}
  P^{bq}(\lambda)\equiv A^{b}{}_{a}S^{aq}(\lambda)
\end{align}
where~$A^{b}{}_{a}$ is given by~$A^{b}{}_{a}\equiv
n^{2}g^{bq}-n^{b}n^{q} +\hat{G}^{ba}n_{a}\tilde{G}^{cq}n_{c}$. On the
other hand, it is possible to demonstrate that
\begin{align}
  \det\left( P^{b}{}_{q}( \lambda) \right) = ( l_{a}g^{ab}l_{b} )^{2}(
  l_{c}\hat{g}^{cd}l_{d}) ( l_{e}\tilde{g}^{ef}l_{f} )
  =l^{4}\hat{l}^{2}\tilde{l}^{2} \label{eq_det_P_1}
\end{align}
which implies that~$A^{b}{}_{a}$ is invertible, since~$n_a$ is
timelike with respect to the three metrics
and~$\det(A^{b}{}_{a})=n^{4}\hat{n}^{2}\tilde{n}^{2}\neq0$.

Additionally, assuming~$l_{a}=\lambda n_{a}+k_{a}$, we can
rewrite~$l.\hat{g}.l $ as
\begin{align}
  l_{a}\hat{g}^{ab}l_{b}=\lambda^{2}\hat{n}^{2}
  +2(  n.\hat{g}.k) \lambda+\hat{k}^{2},\nonumber\\
  =\hat{n}^{2}( \lambda-\hat{\lambda}_{+} )  ( \lambda -\hat{\lambda}_{+} ),
\end{align}
and the same holds for~$l.g.l$ and~$l.\tilde{g}.l $, where
\begin{align}
  \lambda_{\pm} & \equiv\frac{-(n.k)\pm\sqrt{(n.k)^{2}-n^{2}k^{2}}}{n^{2}},\\
  \hat{\lambda}_{\pm}  &  \equiv\frac{-\left(  n.\hat{g}.k\right)  \pm\sqrt{\left(
      n.\hat{g}.k\right)  ^{2}-\hat{n}^{2}\hat{k}^{2}}}{\hat{n}^{2}},\\
  \tilde{\lambda}_{\pm}  &  \equiv \frac{-(n.\tilde{g}.k) \pm
    \sqrt{( n.\tilde{g}.k)^{2}-\tilde{n}^{2}\tilde{k}^{2}}}
	{\tilde{n}^{2}}.  \label{eigenva_1}
\end{align}
Finally, the determinant of~$S^{b}{}_q(\lambda)$ is
\begin{align}
	\det\left( S^{b}{}_{q}( \lambda ) \right) =
	( \lambda-\lambda_{+})^{2}
        ( \lambda-\lambda_{-} )^{2}
	( \lambda-\hat{\lambda}_{+}) 
	( \lambda-\hat{\lambda}_{-})
	( \lambda-\tilde{\lambda}_{+}) 
	( \lambda-\tilde{\lambda}_{-})
        \label{eq_det_ele_1}%
\end{align}

These eigenvalues~$\lambda_{\pm}$, $\hat{\lambda}_{\pm}$
and~$\tilde{\lambda}_{\pm}$ must be real in order to satisfy condition
1 in Definition~\ref{def_strong_hyp_1}. This requirement implies that
all three metrics must be Lorentzian, as we explain next.

Since~$k_a$ is spacelike with respect to~$g^{ab}$ and orthogonal
to~$n_a$, and~$n_a$ is timelike with respect to the three metrics, we
can, without loss of generality, assume that~$k_a$ is spacelike with
respect to all three metrics. This leads to the
conditions~$-n^{2}k^{2} \geq 0$, $-\hat{n}^{2} \hat{k}^{2} \geq 0$,
and~$-\tilde{n}^{2} \tilde{k}^{2} \geq 0$, ensuring that the
eigenvalues~$\lambda_{\pm}$, $\hat{\lambda}_{\pm}$,
and~$\tilde{\lambda}_{\pm}$ are real, thereby satisfying condition 1
in Definition~\ref{def_strong_hyp_1}. It is evident that this
condition fails if at least one of these metrics has the
signature~$(+,+,+,+)$.

Furthermore, suppose that~$\hat{g}^{ab}$ (or any of the other metrics)
has the signature~$(-,+,+,0)$. Under this assumption, there exists a
vector~$k_a$ such that~$\hat{g}^{ab} k_b = 0$, leading to a degenerate
eigenvalue~$\hat{\lambda}_{\pm} = 0$. As will become evident from the
explicit expressions of the eigenvectors, whenever such a degeneracy
occurs, condition 2 of Definition~\ref{def_strong_hyp_1} fails.

%%%%%%%%%%%%%%%%%%%%%%%%%%%%%%%%%%%%%%%%%%%%%%%%%%%%%%%%%%%%%
\subsubsection{Characteristic structure and strong hyperbolicity}
%%%%%%%%%%%%%%%%%%%%%%%%%%%%%%%%%%%%%%%%%%%%%%%%%%%%%%%%%%%%%

In this subsection we show the eigenpairs~$\left(\lambda_{i},(
v_{i})_{q}\right)$ of~$P^{bq}(\lambda)$ i.e. the pairs that
satisfied~$P^{bq}( \lambda_{i})( v_{i})_{q}=0$. We also check when the
algebraic and geometric multiplicities of each eigenvalue of~$P^{bq}(
\lambda)$ are equal or different, thereby determining whether
condition 2 of definition~\ref{def_strong_hyp_1} holds.

We begin by introducing useful notation for the null vectors of the
three metrics. Using Eqns.~\eqref{eigenva_1}, we define
\begin{align}
  ( l_{\pm} )_{a} &\equiv\lambda_{\pm}n_{a}
  +k_{a}\rightarrow(l_{\pm})_{a}g^{ab}(l_{\pm})_{b}=0,\nonumber\\
  ( l_{\hat{\pm}} )_{a} &\equiv\hat{\lambda}_{\pm}n_{a}
  +k_{a}\rightarrow(l_{\hat{\pm}})_{a}\hat{g}^{ab}(l_{\hat{\pm}})_{b}=0,
  \nonumber\\
  (  l_{\tilde{\pm}} )_{a} &\equiv\tilde{\lambda}_{\pm}n_{a}
  +k_{a}\rightarrow( l_{\tilde{\pm}} )_{a}\tilde{g}^{ab}(
  l_{\tilde{\pm}})_{b}=0.\label{eq_null_g_1}
\end{align}
This means that for each fixed normalized~$k_{a}$, the
metrics~$g^{ab}$, $\hat{g}^{ab}$ and~$\tilde{g}^{ab}$ have two null
vector, $( l_{\pm})_{a}$, $( l_{\hat{\pm}})_{a}$,
$(l_{\tilde{\pm}})_{a}$ respectively. In addition, we notice that the
null cones have intersections when the~$\lambda$'s in these null
vectors become equals. We also introduce the normalize
vectors~$(e_{1,2})_{q}$, they satisfy~$e_{1,2}^2=1$
and~$e_{1}.e_{2}=e_{1,2}.n=e_{1,2}.k=0$.

%%%%%%%%%%%%%%%%%%%%%%%%%%%%%%%%%%%%%%%%%%%%%%%%%%%%%%%%%%%%%
\begin{figure}[t!] 
%%%%%%%%%%%%%%%%%%%%%%%%%%%%%%%%%%%%%%%%%%%%%%%%%%%%%%%%%%%%%
  \centering \includegraphics[width=1\textwidth]{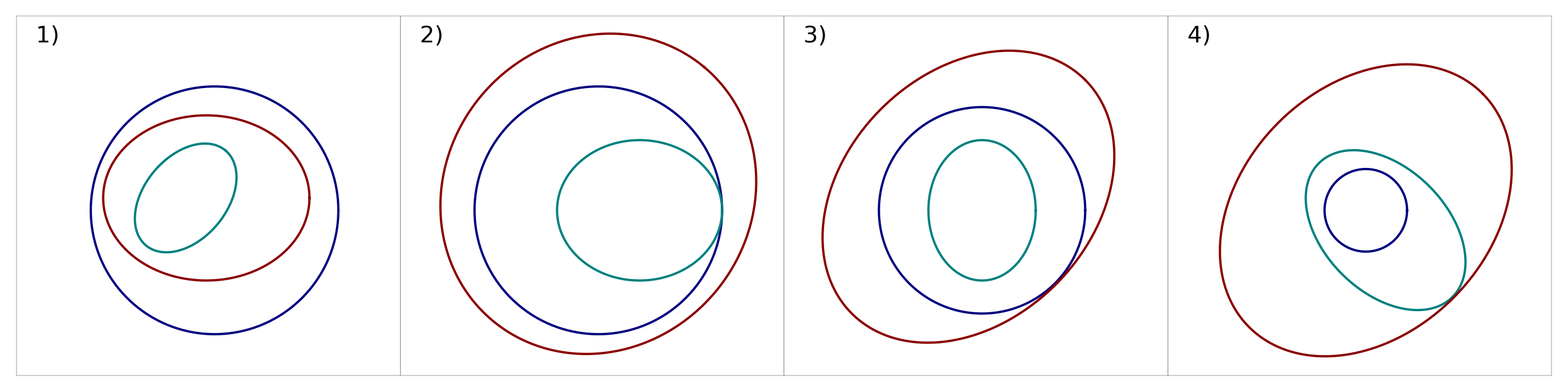}
	\caption{ In these figures, we present the four different
          cases (\ref{caso_1}), (\ref{caso_2}), (\ref{caso_4}), and
          (\ref{caso_6}). The plots depict some ``horizontal''
          projection of the null cones associated with the metrics:
          $\textcolor{navy}{g_{ab}}$ (\textcolor{navy}{blue}),
          $\textcolor{teal}{\tilde{g}_{ab}}$
          (\textcolor{teal}{green}), and
          $\textcolor{DarkRed}{\hat{g}_{ab}}$
          (\textcolor{DarkRed}{red}). In case (1), the null cones do
          not intersect; in case (2), there is a single intersection
          between the null cones of $\textcolor{navy}{g_{ab}}$ and
          $\textcolor{teal}{\tilde{g}_{ab}}$; in case (3), the
          intersection occurs between the null cones of
          $\textcolor{navy}{g_{ab}}$ and
          $\textcolor{DarkRed}{\hat{g}_{ab}}$; and in case (4),
          between those of $\textcolor{teal}{\tilde{g}_{ab}}$ and
          $\textcolor{DarkRed}{\hat{g}_{ab}}$. Under certain
          conditions explained in the text, cases (1), (2), and (3)
          are strongly hyperbolic, whereas case (4) is intrinsically
          weakly hyperbolic.  We emphasize that, since our discussion
          focuses on hyperbolicity, we do not impose causality
          restrictions. As a result, the null cones of
          $\textcolor{teal}{\tilde{g}_{ab}}$ and
          $\textcolor{DarkRed}{\hat{g}_{ab}}$ can be larger than the
          spacetime null cone of $\textcolor{navy}{g_{ab}}$, as shown
          in the figure.  }
	\label{fig:mi_imagen}
\end{figure}
%%%%%%%%%%%%%%%%%%%%%%%%%%%%%%%%%%%%%%%%%%%%%%%%%%%%%%%%%%%%%

In the following four cases, we analyze conditions 2 and 3 from
Definition~\ref{def_strong_hyp_1}. The first three cases satisfy
condition 2 and, under certain additional assumptions, also satisfy
condition 3, resulting in strongly hyperbolic systems. The fourth
case, however, does not satisfy condition 2, making the system
intrinsically weakly hyperbolic.

In Figure 1, we illustrate the projections of the null cones
corresponding to each of these cases. It is evident that any other
configuration (when the null cones intersect more than once), can be
derived from one of these four fundamental cases.

\begin{case}
  \label{caso_1} Null cones of~$g^{ab}$, $\tilde{g}^{ab}$
  and~$\hat{g}^{ab}$ has not intersection so that for each~$k_{a}$,
  the eigenvalues~$\lambda_{+},\tilde{\lambda}_{+},\hat{\lambda}_{+}$,
  $\lambda_{-},\tilde{\lambda}_{-},\hat{\lambda}_{-}$ are distinct.
\end{case}

From Eqn.~\eqref{eq_det_ele_1}, the algebraic multiplicities are
\begin{align}
q_{\lambda_{+}}=2, & \ q_{\tilde{\lambda}_{+}}=1, &  q_{\hat{\lambda}_{+}}=1,\nonumber\\
q_{\lambda_{-}}=2, & \ q_{\tilde{\lambda}_{-}}=1, & q_{\hat{\lambda}_{-}}=1.
\end{align}

The eigenpairs are
\begin{align}
( \lambda_{+},(v_{1})_{a}) & \
( \lambda_{+},(v_{2})_{a}) &
( \tilde{\lambda}_{+}, (v_{3})_{a}) & \
( \hat{\lambda}_{+},(v_{4})_{a}) \nonumber\\
( \lambda_{-},(w_{1})_{a}) & \
( \lambda_{-},(w_{2})_{a}) &
( \tilde{\lambda}_{-},(w_{3})_{a}) & \ ( \hat{\lambda}_{-},(w_{4})_{a})
\end{align}
where
\begin{align}
  (v_{1})_{q} &= (e_{1})_{q}-\frac{l_{+}.\tilde{G}.e_{1}}
  { l_{+}.\tilde{g}.l_{+} }(l_{+})_{q},
  &\quad
  (w_{1})_{q} &= (e_{1})_{q}-\frac{  l_{-}.\tilde{G}.e_{1} }
  {l_{-}.\tilde{g}.l_{-}}(l_{-})_{q},  \label{vectors_v_w_1}\\
  (v_{2})_{q} &= (e_{2})_{q}-\frac{l_{+}.\tilde{G}.e_{2}}
  {l_{+}.\tilde{g}.l_{+} }(l_{+})_{q},
  &\quad
  (w_{2})_{q} &= (e_{2})_{q}-\frac{l_{-}.\tilde{G}.e_{2}}
  {l_{-}.\tilde{g}.l_{-}}(l_{-})_{q},\\
  (v_{3})_{q} &= (l_{\tilde{+}})_{q},
  &\quad
  (w_{3})_{q} &= ( l_{\tilde{-}})_{q},\\
  (v_{4})_{q} &= g_{qc}\hat{G}_{+}^{cd}(l_{\hat{+}})_{d},
  &\quad (w_{4})_{q} &= g_{qc}\hat{G}_{-}^{cd}(l_{\hat{-}})_{d},
  \label{vectors_v_w_4}
\end{align}
with
\begin{align}
  \begin{array}[c]{ccc}
    \hat{G}_{+}^{cd} \equiv\left(\hat{\beta}_{+}\hat{G}^{cd}-\hat{\alpha}_{+}
    g^{cd}\right),  &  & \hat{G}_{-}^{cd}\equiv\left(\hat{\beta}_{-}\hat{G}^{cd}
    -\hat{\alpha}_{-}g^{cd}\right),
  \end{array}
\end{align}
and
\begin{align}
\begin{array}[c]{ccc}
  \hat{\alpha}_{\pm}\equiv
  (l_{\hat{\pm}})_{a} (g^{ab}+\tilde{G}^{ac}g_{cd}\hat{G}^{db})
  (l_{\hat{\pm}})_{b}, & &
  \hat{\beta}_{\pm}\equiv
  (l_{\hat{\pm}})_{a}\tilde{g}^{ab}
  (l_{\hat{\pm}})_{b}\,.
\end{array}
\end{align}
We conclude that
\begin{align}
  s_{\lambda_{+}}&\equiv
  \dim\left( \ker(P^{b}{}_{q}(\lambda_{+}))\right) =2, &
  s_{\tilde{\lambda}_{+}}\equiv
  \dim\left(\ker( P^{b}{}_q(\tilde{\lambda}_{+}))\right) =1,\nonumber\\
  s_{\lambda_{-}}&\equiv
  \dim\left(\ker(P^{b}{}_q(  \lambda_{-}))\right) = 2, &
  s_{\tilde{\lambda}_{-}}\equiv
  \dim\left(\ker( P^{b}{}_q(\tilde{\lambda}_{-}))\right) = 1,
\end{align}
and
\begin{align}
\begin{array}[c]{c}
  s_{\hat{\lambda}_{+}}\equiv
  \dim\left( \ker(P^{b}{}_q(\hat{\lambda}_{+}))\right) = 1,\\
  s_{\hat{\lambda}_{-}}\equiv
  \dim\left( \ker(P^{b}{}_q(\hat{\lambda}_{-}))\right) = 1\,.
\end{array}
\end{align}
Finally, since~$s_{i}=q_{i}$ for all eigenvalues, condition~2 is
satisfied. In addition, since the eigenvectors are continuous with
respect to~$k_a$ condition 3 is satisfied too and the systems is
strongly hyperbolic.

As mentioned before, if the signature of~$\hat{g}^{ab}$
(or~$\tilde{g}^{ab}$) is~$( -,+,+,0)$, then there exist~$k_{a}$ such
that~$\tilde{g}^{ab}k_b=0$. Consequently, $\hat{\lambda}_{\pm }=0$
 increases the algebraic multiplicity
to~$q_{\hat{\lambda}_{+}}=2$, while the geometric multiplicity
remains~$q_{\hat{\lambda}_{+}}=1$. This occurs because the
eigenvectors~$(v_{4})_{q}$ and~$( w_{4})_{q}$ become equals, as~$(
l_{\hat{+}})_{q}=(l_{\hat{-}})_{q}=k_q$. This degeneration results in
the failure of condition 2 in definition~\ref{def_strong_hyp_1}.

\begin{case}
  \label{caso_2} $g^{ab}$ and~$\tilde{g}^{ab}$ share one null vector,
  and they do not share null vectors with~$\hat{g}^{ab}$, so that for
  a particular value of~$k_{a}=(k_{\textrm{deg}})_{a}$, the
  eigenvalues satisfy~$\lambda_{+}=\tilde{\lambda}_{+}\neq
  \hat{\lambda}_{+}$.
\end{case}

In this case, algebraic multiplicities are
\begin{align}
q_{\lambda_{+}=\tilde{\lambda}_{+}}&=3, & q_{\hat{\lambda}_{+}}&=1,\nonumber\\
q_{\lambda_{-}}&=2, & q_{\tilde{\lambda}_{-}}&=1, & q_{\hat{\lambda}_{-}}=1\,.
\end{align}

Eigenvectors are the same to the previous case, maybe except
for $v_{1}$, $v_{2}$ and~$v_{3}$, those associated
to $\lambda_{+}=\tilde{\lambda}_{+}$.

We notice that when~$\lambda=\lambda_{+}$, the covector~$(l_{+}) _{a}$
is null~$l_{+}^{2}=0$, and the characteristic equation reduce to
\begin{equation}
  P^{bq}\delta A_{q}= \left( -(l_{+})^{b} (l_{+})^{q}
  +\hat{G}^{ba}(l_{+})_{a}\tilde{G}^{cq}(l_{+})_{c}\right)
  \delta A_q=0\,. \label{sim_red_a_1}
\end{equation}
Since~$q_{\lambda_{+}=\tilde{\lambda}_{+}}=3$ we need to find three
linearly independent solutions of this equation to satisfy
condition~$2$. It is no complicated to conclude that we need to
include one extra conditions in order that this happens. We have two
options:

The first one is
\begin{align}
  \tilde{G}^{cq}(l_{+})_{c}=\alpha(l_{+})^{q},
  \label{caso3_cond_1_a}
\end{align}
in this case the characteristic equations reduce to
\begin{align}
  P^{bq}\delta A_{q}
  =\left( -(l_{+})^{b}+\alpha\hat{G}^{ba}(l_{+})_{a}\right)
  (l_{+})^{q}\delta A_{q}=0, \label{sim_red_b_2}
\end{align}
so the three solution for~$\delta A_{q}$ can be obtained from
Eqns.~\eqref{vectors_v_w_1}-\eqref{vectors_v_w_4}. They are
\begin{align}
(v_{1})_{q}=(e_{1})_{q},\quad
(v_{2})_{q}=(e_{2})_{q},\quad
(v_{3})_{q}=(l_{+})_{q},
\end{align}
so condition~$2$ is satisfied. In addition, noticing that~$v_1,$ $v_2$
and~$v_3$ are continuous for all~$k_a$, since the expression for them
Eqns.~\eqref{vectors_v_w_1}-\eqref{vectors_v_w_4} do not change for
any value of~$k_a$, we conclude that condition 3 holds and that this
system is strongly hyperbolic.

The other option is
\begin{align}
  \hat{G}^{ba}(l_{+})_{a}=\gamma(l_{+})^{b},\label{caso3_cond_2_b}
\end{align}
in this case, the characteristic equation reduce to
\begin{align}
  (l_{+})^{b} (l_{+})_c \left( -g^{cq}+\gamma\tilde{G}^{cq}\right)
  \delta A_{q}=0, \label{cha_eq_A_2}
\end{align}
so it is always possible to find three solutions for~$\delta
A_{q}$. However we need the explicit expression of~$\tilde{G}^{ba}$ to
calculate them, and finally check condition 3. We can not give a
general statement about the strong hyperbolicity of this case,
until~$\tilde{G}^{ba}$ is given.

\begin{case}
  \label{caso_4} $g^{ab}$ and~$\hat{g}^{ab}$ share one null vector,
  and they do not share null vectors with~$\tilde{g}^{ab}$, i.e. for a
  particular value of~$k_{a}=(k_{\textrm{deg}})_{a}$, the eigenvalues
  satisfy $\lambda_{+}=\hat{\lambda}_{+} \neq\tilde{\lambda}_{+}$.
\end{case}

In this case, algebraic multiplicities are
\begin{align}
q_{\lambda_{+}=\hat{\lambda}_{+}}&=3, & q_{\tilde{\lambda}_{+}}&=1, & \\
q_{\lambda_{-}}&=2, & q_{\tilde{\lambda}_{-}}&=1, & q_{\hat{\lambda}_{-}}=1\,. &
\end{align}

This case is similar to the previous one, the eigenvectors are equal
to case~(\ref{caso_1}), maybe except for~$v_{1}$, $v_{2}$ and~$v_{4}$,
those associated to~$\lambda_{+}=\hat{\lambda}_{+}$.

As before, if we assume condition~(\ref{caso3_cond_1_a}), we obtain
the characteristic equation~(\ref{sim_red_a_1}) and three solution
for~$\delta A_q$, satisfying condition 2. We can obtain these
solutions from Eqns.~\eqref{vectors_v_w_1}-\eqref{vectors_v_w_4}, but
we can not know if they are linearly independent. Therefore, if we can
check that
\begin{align}
  (v_{1})_{q} &= (e_{1})_{q},\\
  (v_{2})_{q} &= (e_{2})_{q},\\
  (v_{4})_{q} &= g_{qc}\hat{G}_{+}^{cd}(l_{\hat{+}})_{d}
  = (l_+.\tilde{g}.l_+)\hat{G}^{cd}(l_{\hat{+}})_{d}\,.
\end{align}
are linearly independent, then by using the same previous continuous
argument condition 3 is satisfied, and the system is strongly
hyperbolic. As before this need to be checked with the explicit
expression of~$\hat{G}^{cd}$. If~$v_4$ is not linearly independent
of~$v_1$, $v_2$, then we can choose~$(v_4)_q=(l_{+})_{q}$, but now we
need to check condition 3, since we lose the continuity argument. As a
final comment of this case, it is not clear for us if it possible to
satisfy condition~\eqref{caso3_cond_1_a}) and
keep~$\lambda_{+}=\hat{\lambda}_{+} \neq\tilde{\lambda}_{+}$.

On the other hand, if we assume condition~\eqref{caso3_cond_2_b}, the
characteristic equation is given by Eqn.~\eqref{cha_eq_A_2}, and as in
the previous case, we obtain three linearly independent solutions
for~$\delta A_q$, satisfying condition~2. Those solutions are
\begin{align}
  (v_{1})_{q}&= (e_{1})_{q}-\frac{(l_{+}.\tilde{G}.e_{1})}{(l_{+}.\tilde{g}.l_{+})}
  (l_{+})_{q},\\
  (v_{2})_{q}&=(e_{2})_{q}-\frac{(l_{+}.\tilde{G}.e_{2})}{(l_{+}.\tilde{g}.l_{+})}
  (l_{+})_{q},\\
  (v_{4})_{q}&= N_{q}-\frac{1+\gamma(l_{+}.\tilde{G}.N)}
  {\gamma(l_{+}.\tilde{G}.l_{+})}(l_{+})_{q},\label{eq_v_caso_4}
\end{align}
where~$N^{q}\equiv-\frac{1}{2k^{2}}(l_{-})^{q}$, which
satisfies~$N^{2}=0$, $N.l=-1$ and~$N.e_{1,2}=0$. Again, we can not use
the continuity argument to check condition 3, since those solutions
are not obtained from
Eqns.~\eqref{vectors_v_w_1}-\eqref{vectors_v_w_4}. The explicit
calculation of condition 3 is rather lengthy, so to avoid it we
propose a different condition. Independently of~$k_a$, we assume that
\begin{align}
  \hat{G}^{ab}=\hat{g}^{ab}\,.
  \label{Ghat_condition_1}
\end{align}
In this case the eigenvector expressions are given by
Eqn.~\eqref{eq_v_caso_4} for all~$k_a$. So conditions~2 and~$3$ are
satisfied, concluding that the system is strongly hyperbolic.

\begin{case}
  \label{caso_6} $\tilde{g}^{ab}$ and~$\hat{g}^{ab}$ share one null
  vector, and they do not share null vectors with~$g^{ab}$, so for a
  particular value of~$k_{a}=(k_{\textrm{deg}})_{a}$, the eigenvalues
  satisfied
  that~$\tilde{\lambda}_{+}=\hat{\lambda}_{+}\neq\lambda_{+}$.
\end{case}

In this case, the algebraic multiplicities are
\begin{align}
  q_{\lambda_{+}}&=2, & q_{\tilde{\lambda}_{+}=\hat{\lambda}_{+}}&=2, & \\
  q_{\lambda_{-}}&=2, & q_{\tilde{\lambda}_{-}}&=1, & q_{\hat{\lambda}_{-}}=1\,. &
\end{align}
The eigenvectors are the same that in case~\ref{caso_1}, except
for~$v_3$ and~$v_4$, those associated
to~$\tilde{\lambda}_{+}=\hat{\lambda}_{+}$. Since~$\hat{\beta}_{+}=\left(
l_{\hat{+}}\right) _{a}\tilde{g}^{ab}\left( l_{\hat{+}}\right)
_{b}=0$, the expressions for~$v_3$ and~$v_4$, using
Eqns.~\eqref{vectors_v_w_1}-\eqref{vectors_v_w_4}, become linearly
dependent
\begin{align}
  (v_{3})_{q}= (l_{\hat{+}})_{q}, \quad
  (v_{4})_{q}=-\hat{\alpha}_{+}(l_{\hat{+}})_{q}
\end{align}

We conclude that the geometric multiplicity associated
with~$\tilde{\lambda}_{+}=\hat{\lambda}_{+}$ is reduced
to~$1$. Finally since~$2=q_{\tilde{\lambda}_{+}=\hat{\lambda}_{+}}\neq
s_{\tilde{\lambda}_{+}=\hat{\lambda}_{+}}=1$ condition 2 is not
satisfied. Thus, this system is intrinsically weakly hyperbolic.

%%%%%%%%%%%%%%%%%%%%%%%%%%%%%%%%%%%%%%%%%%%%%%%%%%%%%%%%%%%%%
\subsubsection{Diagonalization of~$\mathbf{M}(\lambda)$
  and~$\mathbf{S}(\lambda)$}
%%%%%%%%%%%%%%%%%%%%%%%%%%%%%%%%%%%%%%%%%%%%%%%%%%%%%%%%%%%%%

In the following discussion, we focus on cases~\ref{caso_1}
and~\ref{caso_6}, providing an explanation of the diagonalization
of~$\mathbf{M}(\lambda)$ and~$\mathbf{S}(\lambda)$. We recall that for
this Electrodynamic example, the first order principal
symbol~$\mathbf{M}(\lambda)$ and second order principal
symbol~$\mathbf{S}(\lambda)$ are matrices of dimension~$8\times8$
and~$4\times4$, respectively.

The matrix~$\mathbf{M}(\lambda)$ can be written as
\begin{align}
\mathbf{M}(\lambda) = \mathbf{P}\left[
\begin{array}
	[c]{cccc}%
	\lambda \mathbf{1}-\mathbf{D}_{P} & 0 & 0 & 0\\
	0 & \lambda \mathbf{1}-\mathbf{D}_{P} & 0 & 0\\
	0 & 0 & \lambda \mathbf{1}-\mathbf{D}_{C} & 0\\
	0 & 0 & X & \lambda \mathbf{1}-\mathbf{D}_{G}%
\end{array}
\right] \mathbf{P}^{-1}
\end{align}
where
\begin{align}
\mathbf{D}_{P}=\left[
\begin{array}
	[c]{cc}%
	\lambda_{+} & 0\\
	0 & \lambda_{-}%
\end{array}
\right],\text{ \ }
\mathbf{D}_{C}=\left[
\begin{array}
	[c]{cc}%
	\hat{\lambda}_{+} & 0\\
	0 & \hat{\lambda}_{-}%
\end{array}
\right],\text{ \ \ }
\mathbf{D}_{G}=\left[
\begin{array}
	[c]{cc}%
	\tilde{\lambda}_{+} & 0\\
	0 & \tilde{\lambda}_{-}%
\end{array}
\right] ,
\end{align}
and~$\mathbf{P}$ is an invertible matrix that depends on~$k_{a}$,
and~$\mathbf{X}$ is a~$2\times2$ matrix.

In case~\ref{caso_1}, the matrix~$\mathbf{P}$ can be computed
using~Eqn.~\eqref{Eq_P_1_a} and~$\mathbf{X}=\mathbf{0}$, it follows
that~$\mathbf{M}(\lambda)$ is diagonalizable. Furthermore,
$\mathbf{S}(\lambda)$ can be written as a product of two
diagonalizable first order pencils. It is
\begin{align} 
  \mathbf{S}(\lambda)= \mathbf{F} \left[
    \begin{array}
      [c]{cc}%
      \lambda \mathbf{1}-\mathbf{D}_{P} & 0\\
      0 & \lambda \mathbf{1}-\mathbf{D}_{C}%
    \end{array}
    \right] \mathbf{F}^{-1}
  \mathbf{V}_{1}
  \left[\begin{array}
      [c]{cc}%
      \lambda \mathbf{1}-\mathbf{D}_{P} & 0\\
      0 & \lambda \mathbf{1}-\mathbf{D}_{G}%
    \end{array}
    \right] \mathbf{V}_{1}^{-1} \label{EM_S_1}%
\end{align}
where
\begin{align}
\mathbf{V}_{1}&=\left[\begin{array}[c]{cccc}%
	(v_{1})_{q} & (w_{1})_{q} & (v_{3})_{q} & (w_{3})_{q}%
  \end{array}\right],\quad
\mathbf{V}_{1}\mathbf{Q}=\left[\begin{array}[c]{cccc}%
  (v_{2})_{q} & (w_{2})_{q} & (v_{4})_{q} & (w_{4})_{q}%
  \end{array}\right],
\end{align}
and
\begin{align}\mathbf{F}=\mathbf{V}_{1}
(\mathbf{D}_{1}\mathbf{Q}-\mathbf{Q}\mathbf{D}_{2})\,.
\end{align}

In case~\ref{caso_6},the matrix~$\mathbf{X}$ is given
by~$\mathbf{X}=\textrm{diag}(1,0)$ which implies
that~$\mathbf{M}(\lambda)$ is not diagonalizable. Interestingly,
$\mathbf{S}(\lambda)$ can still be expressed as a product of two first
order diagonzalible pencils as Eqn.~\eqref{EM_S_1}, where
\begin{align}
  \mathbf{V}_{1}&=\left[
    \begin{array}
      [c]{cccc}%
      (e_{1})_{q}-\frac{(l_{+}.\tilde{G}.e_{1})
      }{(l_{+}.\tilde{g}.l_{+})}(l_{+})  _{q} & (e_{2})_{q}
      -\frac{(l_{-}.\tilde{G}.e_{2})}{l_{-}.\tilde{g}.l_{-}}(l_{-})_{q} &
      (l_{\tilde{+}})_{q} & (l_{\tilde{-}})_{q}
    \end{array}\right] \nonumber\\
  \mathbf{V}_{1}\mathbf{Q} & =\left[\begin{array}[c]{cccc}%
      (  e_{1})_{q}-\frac{(l_{+}.\tilde{G}.e_{1})}{(l_{+}.\tilde{g}.l_{+})}
      (l_{+})_{q} & (e_{2})_{q}-\frac{(l_{-}.\tilde{G}.e_{2})}{l_{-}
	.\tilde{g}.l_{-}}(  l_{-})_{q} & (l_{\tilde{+}})_{q} & (l_{\tilde{-}})_{q}
    \end{array}
    \right].
\end{align}
However, in this case, $\mathbf{F}\neq\mathbf{V}_{1}
(\mathbf{D}_{1}\mathbf{Q}-\mathbf{Q}\mathbf{D}_{2})$. In fact, it can
be verified that~$\det\left(\mathbf{V}_{1}
(\mathbf{D}_{1}\mathbf{Q}-\mathbf{Q}\mathbf{D}_{2})\right)=0$.

%%%%%%%%%%%%%%%%%%%%%%%%%%%%%%%%%%%%%%%%%%%%%%%%%%%%%%%%%%%%%
\section{Conclusions}
\label{section:Conclusions}
%%%%%%%%%%%%%%%%%%%%%%%%%%%%%%%%%%%%%%%%%%%%%%%%%%%%%%%%%%%%%

In the context of gravitational wave science there is interest in
considering theories that deviate from GR. Such considerations result
in systems whose characteristic structure may be much more complicated
than that of GR. It is therefore valuable to have characterizations of
hyperbolicity, which can in turn guarantee well-posedness of the
initial value problem, that can be applied as easily as possible. So
motivated, in this study we provided a detailed explanation of how to
verify the strong hyperbolicity of fully second-order PDEs (see
Definition~\ref{def_strong_hyp_1}) without first explicitly reducing
them to a first-order system. We demonstrated that analyzing the
kernel of the second-order pencil (and principal symbol)
$\mathbf{S}(\lambda)$ is sufficient to establish strong hyperbolicity
of a formulation of a theory. In Theorem~\ref{theor_2}, we summarized
how to deduce the kernel of the first-order
pencil~$\mathbf{M}(\lambda)$ (obtained by reducing the second-order
PDEs to a first-order system) given the kernel of
$\mathbf{S}(\lambda)$. By employing matrix pencils, we were moreover
able to preserve the covariance of the equations throughout all
calculations, thereby avoiding the need to introduce a~$3+1$
decomposition from the outset.

In the theorem mentioned above the kernel of~$\mathbf{S}(\lambda)$ is
accommodated in the columns of two square matrices, $\mathbf{V}_1$
and~$\mathbf{V}_1\mathbf{Q}$, where~$\mathbf{V}_1$ is invertible
while~$\mathbf{Q}$ is not necessarily so. Then, to establish the
diagonalization of~$\mathbf{M}(\lambda)$, the most critical condition
required to ensure strong hyperbolicity, the
matrix~$\mathbf{D}_1\mathbf{Q}-\mathbf{Q}\mathbf{D}_2$ must be
invertible, where~$\mathbf{D}_1$ and~$\mathbf{D}_2$ are real diagonal
matrices composed of the eigenvalues of~$\mathbf{S}(\lambda)$. Under
these conditions, Theorem~\ref{theor_M_S} reveals an important result
regarding the structure of~$\mathbf{S}(\lambda)$. Specifically, the
second-order pencil~$\mathbf{S}(\lambda)$ can be factorized as the
product of two first-order diagonalizable pencils whose explicit
expressions are provided in the theorem. However, it should be
emphasized that the converse is not always true: factorization alone
does not imply strong hyperbolicity. This limitation is explicitly
demonstrated in the examples presented in this article. In particular,
repeated application of a strongly hyperbolic first order operator
results in a second order system that is only weakly hyperbolic. The
latter result is summarized by Theorem~\ref{theorem_D2}.

To illustrate the application of our definition of strong
hyperbolicity, we presented two examples. In the first, the
\textit{almost wave equation in 1+1}, we presented a straightforward
application of Theorems~\ref{theor_2} and~\ref{theor_M_S}. We
concluded that, in the non-degenerate case ($a \neq b$), the PDE is
strongly hyperbolic, whereas in the degenerate case ($a = b$), it is
weakly hyperbolic. This example provides an clarifying discussion on
how the first-order reduction framework aligns with our definition of
strong hyperbolicity.  The degenerate case~$(
\partial_{t}-a\partial_{x})^{2}\phi=0$ shares structural similarities
with the examples studied in~\cite{FigHelKov24}, where the weakly
hyperbolic equation~$\square^{2}\phi=f$ is analyzed, and strongly
hyperbolic first order reductions can only be arrived at by performing
a careful first order reduction in which one avoids introducing all
subdominant derivatives as evolution variables. It would be very
interesting to investigate such reductions systematically for generic
weakly hyperbolic PDEs of both low and high derivative order. With
that, both higher order derivatives as in~\cite{FigHelKov24}, and
additional fields, as in~\cite{SalCloFig22} could be naturally
treated. This is left for the future. For now, we simply observe that
Theorem~\ref{theorem_D2} guarantees that such a careful first order
reduction or equivalently, a carefully chosen norm is needed for
second order systems with this `repeated-operator' structure, and is
very indicative of a similar statement for higher order systems
constructed from repeated hyperbolic operators.

In the second example, we used our formalism to treat Maxwell's
equations in terms of the vector potential. We considered a general
formulation of the extension and gauge fixing, where two extra
Lorentzian metrics~$\tilde{g}^{ab}$ and $\hat{g}^{ab}$ were introduced
to fix the gauge and to grant specific formal evolution equations to
the constraints. We examined the strong hyperbolicity of these PDEs,
including cases where the characteristics are degenerate, not
previously studied in the literature, and explicitly computed the
eigenvectors associated with the system. So, in this formulation, we
demonstrated that if~$\tilde{g}^{ab}$ and~$\hat{g}^{ab}$ share at
least one common null covector that is not null with respect
to~$g^{ab}$, then the system is weakly hyperbolic. In all other cases,
as long as certain additional conditions hold, the PDEs are strongly
hyperbolic. Estimates for solutions could be easily derived in a
general formalism from the first-order reductions. The use of energy
estimates can clarify the discussion of the first example in the
degenerate case~$a=b$.

To wrap up, we have seen that the characterization of hyperbolicity of
fully second-order systems via matrix pencils offers an efficient and
practical approach in applications. We are currently extending this
framework to systems involving higher-order derivatives, and, in
parallel, we are conducting an investigation of high-order systems
featuring repeated hyperbolic operators, as well as gauge and
constraint degrees of freedom.

%%%%%%%%%%%%%%%%%%%%%%%%%%%%%%%%%%%%%%%%%%%%%%%%%%%%%%%%%%%%%
\section*{Acknowledgements}
%%%%%%%%%%%%%%%%%%%%%%%%%%%%%%%%%%%%%%%%%%%%%%%%%%%%%%%%%%%%%

This work was partially supported by the project PID2022-138963NB-I00,
funded by the Spanish Ministry of Science, Innovation and Universities
(MCIN/AEI/10.13039/501100011033); and by FCT (Portugal) through grant
Numbers UID/00099/2025, UID/PRR/00099/2025 and 2023.12549.PEX. We are
grateful to the ``Programa de projectes de recerca amb investigadors
convidats de prestigi reconegut,” which supported DH's UIB visit.

%%%%%%%%%%%%%%%%%%%%%%%%%%%%%%%%%%%%%%%%%%%%%%%%%%%%%%%%%%%%%
\bibliographystyle{unsrt}
\bibliography{StrongSecond.bbl}
%%%%%%%%%%%%%%%%%%%%%%%%%%%%%%%%%%%%%%%%%%%%%%%%%%%%%%%%%%%%%

%%%%%%%%%%%%%%%%%%%%%%%%%%%%%%%%%%%%%%%%%%%%%%%%%%%%%%%%%%%%%
\end{document}